\documentclass[reqno,11pt]{amsart}
\usepackage{latexsym,amsmath,amscd,amssymb,graphics,mathrsfs}
\usepackage{amsthm}
\usepackage{enumerate}
\usepackage{fullpage}
\usepackage[all]{xy}
\usepackage{hyperref}

\newcommand \al{\alpha}
\newcommand\be{\beta}
\newcommand\ga{\gamma}

\newcommand\ze{\zeta}

\renewcommand\th{\theta}
\newcommand\io{\iota}

\newcommand\la{\lambda}

\newcommand\si{\sigma}

\newcommand\ph{\varphi}

\newcommand\om{\omega}
\newcommand\Ga{\Gamma}

\newcommand\La{\Lambda}

\newcommand\Om{\Omega}

\newcommand\ie{\ie, }
\newcommand\eg{e.g.\ }

\newcommand\oo{{\infty}}

\renewcommand\o{\circ}

\newcommand\x{\times}
\newcommand\on{\operatorname}
\newcommand\Ad{\on{Ad}}

\newcommand\ori{\textnormal{or}}
\newcommand\Emb{\on{Emb}}

\newcommand\Ann{\on{Ann}}

\newcommand\ex{\on{ex}}

\newcommand\Den{\on{Den}}

\newcommand\iso{\on{iso}}

\newcommand\ann{\on{ann}}
\newcommand\aug{\on{aug}}
\newcommand\hnull{\on{hnull}}
\newcommand\deco{\on{deco}}
\newcommand\Aut{\on{Aut}}

\newcommand\D{\mathcal{D}}
\newcommand\M{\mathcal{M}}

\newcommand\Diff{\on{Diff}}

\newcommand\id{\on{id}}

\newcommand{\symp}{\on{symp}}

\newcommand\reg{\on{reg}}

\renewcommand\L{\on{pounds}}

\newcommand\Ham{\on{Ham}}

\newcommand\KKS{\on{KKS}}

\newcommand\codim{\on{codim}}
\newcommand\coker{\on{coker}}
\newcommand\ham{\on{\mathfrak{ham}}}

\newcommand\GL{\on{GL}}

\newcommand\g{\mathfrak g}

\newcommand\Gr{\on{Gr}}
\newcommand\ou{\mathfrak o}

\newcommand\RR{\mathbb{R}}
\newcommand\R{\mathbb{R}}

\newcommand\X{\mathfrak X}
\newcommand\DD{\mathfrak D}

\newcommand\G{\mathcal G}
\newcommand\A{\mathfrak A}
\renewcommand\L{\mathcal L}

\DeclareMathOperator{\SAut}{SAut}
\newcommand\catclosed{\mathfrak{Mf}}
\DeclareMathOperator{\Bij}{Bij}

\DeclareMathOperator{\rk}{rk}

\newcommand\cataug{\mathfrak{Gr}}

\newcommand\wt{\on{wt}}

\newcommand\hwt{\on{hwt}}

\DeclareMathOperator{\evol}{evol}
\DeclareMathOperator{\Evol}{Evol}

\newtheorem{theorem}{Theorem}[section]

\newtheorem{lemma}[theorem]{Lemma}
\newtheorem{proposition}[theorem]{Proposition}
\newtheorem{corollary}[theorem]{Corollary}
\theoremstyle{remark}
\newtheorem{remark}[theorem]{Remark}
\newtheorem{example}[theorem]{Example}

\begin{document}

\title{Nonlinear Grassmannians: plain, decorated, augmented}

\author{Stefan Haller}
\address{Stefan Haller, Department of Mathematics, University of Vienna, Oskar-Morgenstern-Platz~1, 1090 Vienna, Austria. \href{https://orcid.org/0000-0002-7064-2215}{ORCID 0000-0002-7064-2215}}
\email{stefan.haller@univie.ac.at}

\author{Cornelia Vizman}
\address{Cornelia Vizman, Department of Mathematics, West University of Timi\c soara, Bd.~V.~Parvan~4, 300223 Timi\c soara, Romania. \href{https://orcid.org/0000-0002-5551-3984}{ORCID 0000-0002-5551-3984}}
\email{cornelia.vizman@e-uvt.ro}


\begin{abstract}
	Decorated and augmented nonlinear Grassmannians can be used to parametrize coadjoint orbits of classical diffeomorphism groups.
	We provide a general framework for decoration and augmentation functors that facilitates the construction of a smooth structure on decorated or augmented nonlinear Grassmannians. 
	This permits to equip the corresponding coadjoint orbits with the structure of a smooth symplectic Fr\'echet manifold. 
	The coadjoint orbits obtained in this way are not new.
	Here, we provide a uniform description of their smooth structures.
\end{abstract}

\maketitle


\section{Introduction}

The four classical infinite-dimensional Lie algebras studied by E.~Cartan are:
the Lie algebras of all vector fields, of symplectic vector fields, of
contact vector fields, and of divergence free vector fields
\cite{cartan,singer}.
We aim at giving a unified approach to certain classes of coadjoint
orbits of the corresponding diffeomorphism groups that were previously studied in
\cite{W90,HV04,L09,HV22,GBV24,HV24}.

These coadjoint orbits are \emph{good} in the sense that they admit a
(unique) Fr\'echet manifold structure such that the coadjoint action is
smooth and admits local smooth sections.
This implies that their (formal) Kirillov--Kostant--Souriau (KKS)
symplectic form is smooth and weakly nondegenerate.
We stress the concept of \emph{local smooth sections} throughout this text.
Therefore, some of the results reproved here are in fact slightly
stronger than the ones available in the literature.

The common feature of the mentioned coadjoint orbits is their parametrizations with nonlinear Grassmannians.
These can be plain Grassmannians, but can also consist of submanifolds endowed with an additional structure.
When the structure lives on the submanifold, we call it a \emph{decoration,}
while when the structure lives in the ambient space, we call it an \emph{augmentation.}

In this article we give the general framework for both, decorations and augmentations,
which facilitates the construction of a smooth structure on decorated or augmented nonlinear Grassmannians.
A decoration functor $\DD$ is a functor from the category of (finite dimensional) closed manifolds and diffeomorphisms to the category of Fr\'echet manifolds, such that the $\Diff(S)$ action on $\DD(S)$ is smooth.
In contrast, an augmentation functor is associated to an action of a Lie group $G$ on a finite dimensional manifold $M$.
It is a functor $\A$ from the category of closed submanifolds of $M$ and diffeomorphisms given by the $G$-action to the category of Fr\'echet manifolds, 
such that the action  of the isotropy subgroup $G_N$ on $\A(N)$ is smooth.

Some of these good coadjoint orbits of diffeomorphism groups can also be obtained via symplectic reduction in dual pairs of momentum maps.
Such dual pairs exist for the Hamiltonian group (the ideal fluid dual pair \cite{MW83}), for the full diffeomorphism group (the EPDiff dual pair \cite{HM05}),
for the contact group (the EPContact dual pair \cite{HV22}), and for the group of (exact) volume preserving diffeomorphisms (the EPDiffvol dual pair \cite{HV24}).
According to a general principle for dual pairs, symplectic reduction for the group acting from the right (a reparametrization action in these cases) yields coadjoint orbits for the group acting from the left (the four diffeomorphism groups mentioned above).

The Lie group structure on many isotropy groups appearing in this context remains elusive.
For instance, to our best knowledge, it is not known if the group of Hamiltonian diffeomorphisms preserving an isotropic (or symplectic) submanifold actually is a Lie subgroup in the group of all Hamiltonian diffeomorphisms.
Similar questions for the group of volume preserving diffeomorphisms and the group of contact diffeomorphisms remain open too.

\subsection*{Content of the paper}
In Section~2 we define \emph{good coadjoint orbits} (as those that possess a smooth  structure such that the coadjoint action is smooth and admits local smooth sections),
and we show that the formally defined KKS symplectic form is smooth. 
In the context of the action of a convenient Lie group $G$ on the nonlinear Grassmannian $\Gr(M)$ that is induced by an action on $M$, we introduce the notions of \emph{good orbit} and \emph{good isotropy} at $N\in\Gr(M)$.
Then we exemplify these for several diffeomorphism groups, \eg the Hamiltonian group has good orbits at isotropic and  symplectic $N$,
the contact group at isotropic $N$, and the full diffeomorphism group and the (exact) volume preserving diffeomorphism group at any $N$.
Isodrasts in the symplectic context describe the Hamiltonian orbits at isotropic $N$,
while isodrasts in the volume form context describe orbits of codimension one submanifolds 
under the action of the group of (exact) volume preserving diffeomorphisms.
As an application we obtain Theorems~\ref{P:symp} and \ref{P:codim2} which describe good coadjoint orbits parametrized by plain nonlinear Grassmannians.

Section~3, devoted to decorated nonlinear Grassmannians, starts with the presentation of the general framework for decorations.
Examples of decoration functors $\DD$ comprise orientations, densities, differential forms, (co)homology classes.
In studying the associated smooth structures, the notions of \emph{good orbit} and \emph{good isotropy} for the $\Diff(S)$ action on $\DD(S)$ are useful. 
Good coadjoint orbits parametrized by decorated nonlinear Grassmannians are described in 
Theorem \ref{T:th1}, namely isotropic submanifolds decorated with volume densities for the Hamiltonian group,
and in Theorem \ref{T:th2}, namely hypersurfaces decorated with nowhere zero closed 1-forms for the group of (exact) volume preserving diffeomorphisms.

In Section~4 we study augmented nonlinear Grassmannians.
The general framework is an augmentation functor $\A$ associated to a $G$ action on $M$.
Associated smooth structures are studied, especially for $G$ orbits in the augmented nonlinear Grassmannians.
For appropriate augmentations, the results are used to parametrize good coadjoint orbits of diffeomorphism groups:
$\Ga((|\La|_N\otimes T^*M|_N)\setminus N)$ in Theorems~\ref{T:full}, \ref{T:non0} and \ref{T:cex},
a quotient space of $\Ga((|\La|_N\otimes T^*M|_N)\setminus N)$ in Theorem \ref{T:[]},
and $\Ga((|\La|_N\otimes L^*|_N)\setminus N)$, with $L$ denoting the contact line bundle, in Theorem \ref{T:contact}.
We end the section with augmentations induced by a decoration functor.

\section{Plain nonlinear Grassmannians}\label{S:two}

Throughout this paper we use the convenient calculus of Kriegl and Michor \cite{KM97} to describe smooth structures on spaces of manifolds and maps.
In this calculus, smooth manifolds are modeled on convenient vector spaces, a class of locally convex vector spaces satisfying a weak completeness assumption.

A submanifold is called \emph{splitting submanifold} if the corresponding (closed) linear subspace in a submanifold chart admits a complement, cf.~\cite[Definition~27.11]{KM97}. 
A subset $A$ in a manifold $M$ is called \emph{(splitting) initial submanifold} if every point $a\in A$ admits an open neighborhood $U$ in $M$ such that $(A\cap U)_a$, the set of all points that can be connected to $a$ through a smooth path lying in $A\cap U$, is a (splitting) submanifold of $U$, cf.~\cite[Definition~2.13]{M08}.
Such a subset $A$ inherits a smooth structure for which the inclusion $A\to M$ is an initial smooth map.

The Lie groups $G$ we consider, in general diffeomorphism groups, are Lie groups in the convenient setting of \cite{KM97}.
A subgroup $H$ in a Lie group $G$ will be called a \emph{splitting Lie subgroup} if it is a splitting submanifold of $G$. 
In this case, $H$ itself is a Lie group (with the induced structure).


\subsection{Transitive actions}
Recall that an action of a Lie group $G$ on a manifold $M$ is said \emph{to admit local smooth sections} if, for every $x\in M$, the map provided by the action, $G\to M$, $g\mapsto g\cdot x$, admits a smooth local right inverse defined in an open neighborhood of $x$.
More explicitly, we require that there exists an open neighborhood $U$ of $x$ in $M$ and a smooth map $u\colon U\to G$ such that $u(y)\cdot x=y$ for all $y\in U$.
We may assume w.l.o.g.\ that $u(x)=e$, the neutral element in $G$.
Clearly, any action which admits local smooth sections is locally and infinitesimally transitive.
(The converse implication is not true for general Fr\'echet manifolds,
due to the lack of a general implicit function theorem.)
In particular, its orbits are open and closed in $M$ and, hence, they consist of several connected components of $M$.

The next lemma is useful in detecting transitive actions on decorated nonlinear Grassmannians,
since these are associated bundles to the principal bundle of embeddings.

\begin{lemma} [{\cite[Lemma~A.1]{HV23}}]\label{before.sec}
	Let $P\to B$ be a principal $G$-bundle endowed with the action of a Lie group $H$ on $P$ that commutes with the principal $G$ action.
	Suppose the structure group $G$ acts on another manifold $Q$, and consider the canonically induced $H$ action on the associated bundle $P\times_GQ$.
	If the $H$ action on $P$ and the $G$ action on $Q$ both admit local smooth sections, then the $H$ action on $P\times_GQ$ admits local smooth sections too.
\end{lemma}

A Lie group is called regular if smooth curves in the Lie algebra integrate to smooth curves in the group in a smooth way \cite{Milnor,KM97,GN}.
This ensures the existence of an exponential map.
All Lie groups of diffeomorphisms are regular (in fact all known smooth Lie groups are regular).

We denote the fundamental vector fields of a smooth $G$ action by $\ze_X$, where $X\in\mathfrak g$, the Lie algebra of $G$.

\begin{lemma}[{\cite[Lemma~A.2]{HV23}}]\label{L:loc.sec}
	Let $G$ be a regular Lie group acting on a smooth manifold $Q$.
	Suppose every point $x_0$ in $Q$ admits an open neighborhood $U$ and a smooth map $\sigma:TQ|_U\to\mathfrak g$ such that
	\begin{equation}\label{E:tsigma}
		\zeta_{\sigma(X)}(x)=X, \text{ for all $x\in U$ and $X\in T_xQ$},
	\end{equation}
	where $\ze:\g\to\X(Q)$ denotes the infinitesimal action.
	Then the $G$ action on $Q$ admits local smooth sections.
\end{lemma}	

\begin{proof}
	We construct a local smooth section using an argument due to Moser \cite[Section~4]{M65}.
	Suppose $c:[0,1]\to U$ is a smooth curve.
	Since $G$ is regular, there exists a unique smooth curve $g=\Evol^r\bigl(\sigma\circ c'\bigr)$ in $G$ starting at $e$ and
	with right logarithmic derivative $\dot g(t)=(\sigma\o c')(t)$, cf.~\cite[Definition~38.4]{KM97}.
	Using \eqref{E:tsigma}, we see that $c'(t)=\zeta_{\dot g(t)}(c(t))$,  hence $c(t)=g(t)\cdot c(0)$.
	We obtain a smooth map
	\[
		s:C^\infty([0,1],U)\to G,\qquad s(c):=g(1)=\evol^r\bigl(\sigma\circ c'\bigr),
	\]
	with the property that $c(1)=s(c)\cdot c(0)$.
	To obtain a smooth section on the neighborhood $U$ of $x_0$, it suffices to compose $s$ with a smooth map 
	$U\to C^\infty([0,1],U)$, $x\mapsto c_x$ satisfying $c_x(0)=x_0$ and $c_x(1)=x$.
	The latter can readily be constructed shrinking $U$ so that it becomes (in a chart) star shaped with center $x_0$.
\end{proof}

The next lemma will be used to prove smoothness of the KKS symplectic form on certain (good) coadjoint orbits of diffeomorphism groups parametrized with nonlinear Grassmannians.

\begin{lemma}\label{lemac}
	Let $G$ be a regular Lie group acting on a smooth manifold $Q$,
	and let $J:Q\to\g^*$ be an injective $G$ equivariant map, where $\g^*$ denotes the space of bounded linear functionals on the Lie algebra $\g$.
	Assume that the $G$ action on $Q$ is smooth, transitive, and admits local smooth sections.
	Then the (formal) pullback of the KKS symplectic form on the coadjoint orbit $J(Q)$, denoted  $\om=J^*\om_{\KKS}$,
	is a $G$ equivariant smooth (weakly) symplectic form on $Q$ characterized by
	\begin{equation}\label{omeg}
		\om_q(\ze_X(q),\ze_Y(q))=\langle J(q),[X,Y]\rangle,
	\end{equation}
	for all $q\in Q$ and $X,Y\in\g$.
	Furthermore, $J$ is an (equivariant) moment map for the (Hamiltonian) $G$ action on $Q$.
\end{lemma}

\begin{proof}
Let us first show that the function
\[
	J^Y\colon Q\to\mathbb R,\qquad J^Y(q)=\langle J(q),Y\rangle
\]
is smooth for each $Y\in\mathfrak g$.
To this end let $q_0\in Q$.
Since the $G$ action on $Q$ admits local smooth sections, 
there exists an open neighborhood $U$ of $q_0$ and a smooth map $u\colon U\to G$ such that $u(q)\cdot q_0=q$ for each $q\in U$.
By the $G$ equivariance of $J$, 
\[
	J^Y(q)=\langle J(u(q)\cdot q_0),Y\rangle=\langle J(q_0),\Ad_{u(q)}^{-1}Y\rangle
\]
is smooth in $q\in U$.
If $g(t)$ is a smooth curve in $G$ with $\dot g(0)=X\in\g$ then the equivariance of $J$ gives
$\tfrac{d}{dt}\big|_0\langle J(q),\Ad_{g(t)}Y\rangle
=\tfrac{d}{dt}\big|_0\langle J(g(t)^{-1}\cdot q),Y\rangle
=\tfrac{d}{dt}\big|_0J^Y(g(t)^{-1}\cdot q)$
and the chain rule yields
\begin{equation}\label{E:omega.alt}
	\langle J(q),[X,Y]\rangle=-d_qJ^Y(\zeta_X(q)).
\end{equation}
Since the $G$ action on $Q$ admits local smooth sections it is, in particular, infinitesimally transitive.
Hence, the fundamental vector fields $\zeta_X(q)$ span the full tangent space at $q$.
Using \eqref{E:omega.alt} we conclude that the 2-form $\om$ in \eqref{omeg} is pointwise well defined.
Moreover, for its kernel we have
\begin{equation}\label{E:ker.omega}
	\ker\om_q=\bigcap_{Y\in\g}\ker d_qJ^Y.
\end{equation}
In order to show that $\om_q$ is nondegenerate, suppose that $\ze_X(q)$ is in the kernel of $\om_q$.
By the regularity of $G$, the smooth curve $q(t)=\exp tX\cdot q$ satisfies $\dot q(t)=\ze_X(q(t))$.
Then, denoting the action of $g\in G$ on $Q$ by $\ell_g$,
\begin{multline*}
	\tfrac{d}{dt}J^Y(q(t))=dJ^Y(\dot q(t))
	=dJ^Y(\ze_X(\exp tX\cdot q))
	=dJ^Y(T\ell_{\exp tX}(\ze_{\Ad_{\exp (-tX)}X}(q)))\\
	=dJ^Y(T\ell_{\exp tX}(\ze_{X}(q)))
	=d(J^Y\o\ell_{\exp tX})(\ze_{X}(q))=dJ^{\Ad_{\exp tX}Y}(\ze_X(q))=0,
\end{multline*}
in view of \eqref{E:ker.omega}.
Thus, $J^Y(q(t))$ is constant in $t$ for all $Y$.
Hence, $J$ is constant along the curve $q(t)$.
As $J$ is injective, this curve must be constant, whence $\ze_X(q)=\dot q(0)=0$.
This shows that each $\omega_q$ is weakly nondegenerate.
The following computation shows that $\om$  is also $G$ invariant:
\begin{align*}
	(\ell_g^*\om)_q(\ze_X(q),\ze_Y(q))&=\om_{g\cdot q}(T\ell_g\ze_X(q),T\ell_g\ze_Y(q))
	=\om_{g\cdot q}(\ze_{\Ad_gX}(g\cdot q),\ze_{\Ad_gX}(g\cdot q))\\
	&=\langle J(g\cdot q),[\Ad_gX,\Ad_gY]\rangle=\langle J(g\cdot q),\Ad_g[X,Y]\rangle\\
	&=\langle J( q),[X,Y]\rangle
	=\om_q(\ze_X(q),\ze_Y(q)).
\end{align*}
Since the orbit map $\ell^q\colon G\to Q$, $\ell^q(g)=g\cdot q$ is $G$ equivariant, 
the pullback $\tilde\om=(\ell^q)^*\om$ is a left invariant, hence smooth, differential 2-form on $G$. 
At the identity it is given by $\tilde\om_e(X,Y)\mapsto \langle J(q),[X,Y]\rangle$, hence $\tilde\om=-d\tilde\th$ where $\tilde\th$ denotes the left invariant smooth differential 1-form on $G$ given by $\tilde\th_e=J(q)\in\g^*$.
Using a local smooth section $u\colon U\to G$, $\ell^q\circ u=\id$, we see that on $U$ the form $\omega=u^*(\ell^q)^*\om=u^*\tilde\om=d(-u^*\tilde\theta)$ is smooth and exact.
It follows that $\om$ is a smooth closed differential 2-form on $Q$.

In conclusion, $(Q,\om)$ is a symplectic manifold.
In view of \eqref{E:omega.alt}, $J^Y$ are smooth Hamiltonian functions for $\ze_Y$ for each $Y\in \g$, 
thus the map $J$ is a moment map for the $G$ action on $Q$.
\end{proof}

\begin{remark}\label{R:unique}
	Suppose $G$ is a Lie group acting smoothly on two manifolds $Q_1$ and $Q_2$ such that both actions admit local smooth sections.
	Then any $G$ equivariant bijection $f:Q_1\to Q_2$ is a diffeomorphism.
	Indeed, if $U$ is an open neighborhood of $x_1$ in $Q_1$, and $u\colon U\to G$ is a local smooth section for the $G$ action,
	i.e., $u(x)\cdot x_1=x$ for all $x\in U$, then $f(x)=f(u(x)\cdot x_1)=u(x)\cdot f(x_1)$ is smooth in $x\in U$.
\end{remark}

A coadjoint orbit $\mathcal O\subseteq\mathfrak g^*$ of a Lie group $G$ is called \emph{good} if it admits a Fr\'echet manifold structure such that the $G$ action on $\mathcal O$ is smooth and admits local smooth sections.
In view of Remark~\ref{R:unique} such a smooth structure on $\mathcal O$ is unique, provided it exists.
Using the inclusion $J:\mathcal O\to\g^*$ in  Lemma~\ref{lemac} we  obtain:

\begin{corollary}\label{C:ooo}
If $\mathcal O$ is a good coadjoint orbit of a regular Lie group, then the KKS form is a smooth weakly nondegenerate symplectic form on $\mathcal O$.
\end{corollary}


\subsection{Diffeomorphism groups preserving a geometric structure}\label{SS:geom}
Suppose $G$ is a Lie group acting smoothly on a finite dimensional manifold $M$.
Typical examples of $G$ we are interested in include: the group of Hamiltonian diffeomorphisms, the group of symplectic diffeomorphisms, the group of (exact) volume preserving diffeomorphisms, the group of contact diffeomorphisms, or the automorphism group of a principal bundle over $M$.
In view of the latter example, we do not require the $G$ action on $M$ to be faithful.

Let $\Gr(M)$ denote the nonlinear Grassmannian of $M$, i.e., the space of all closed smooth submanifolds $N$ in $M$.
This is a smooth Fr\'echet manifold with tangent space $T_N\Gr(M)=\Ga(TM|_N/TN)$, see \cite{BF81,M80b,M80c} and \cite[Theorem~44.1]{KM97}. 

For $N\in\Gr(M)$ we let $G\cdot N$ denote the $G$ orbit through $N$ in $\Gr(M)$ and we write $G_N=\{g\in G:g\cdot N=N\}$ for the isotropy group at $N$.
We say \emph{$G$ has good orbit at $N$} if $G\cdot N$ is an initial splitting smooth submanifold in $\Gr(M)$ and the (smooth) $G$ action on $G\cdot N$ admits local smooth sections.
We say \emph{$G$ has good isotropy at $N$} if $G_N$ is a splitting Lie subgroup of $G$.

\begin{lemma}\label{L:geom1}
	If $G$ has good orbit and good isotropy at $N\in\Gr(M)$, then the $G$ equivariant orbit map
	\begin{equation}\label{E:princ.sigma}
		G\to G\cdot N,\quad g\mapsto g\cdot N
	\end{equation}
	is a smooth principal bundle with structure group $G_N$.
\end{lemma}

\begin{proof}
	Suppose $N_0\in G\cdot N$.
	Since the $G$ action on $G\cdot N$  admits local smooth sections, there exists an open neighborhood $U$ of $N_0$ in $G\cdot N$ and a smooth map $u\colon U\to G$ such that $u(\tilde N)\cdot N=\tilde N$ for all $\tilde N\in U$.
	This permits to write down a local trivialization of \eqref{E:princ.sigma} near $N_0$.
	Indeed, 
	\begin{equation}\label{E:trivi}
		 U\times G_N\to G|_{U},\quad(\tilde N,g)\mapsto u(\tilde N)g
	\end{equation}
	is an $G_N$ equivariant diffeomorphism with inverse map
	\[
		g\mapsto(g\cdot N,u(g\cdot N)^{-1}g),
	\]
	where the right hand side in \eqref{E:trivi} denotes the preimage of $U$ under \eqref{E:princ.sigma}.
\end{proof}
	
We say \emph{$G$ has good extension at $N$} if the canonical homomorphism
\begin{equation}\label{E:homo.sigma}
	G_N\to\Diff(N),\quad g\mapsto g|_N 
\end{equation}
admits a local smooth right inverse in a neighborhood of the identity.
Since we do not want to presume good isotropy at $N$, the local section is assumed to be smooth when considered as a map into $G$.
Clearly, $G$ has good extension at $N$ if and only if the $G_N$ action on $\Diff(S,N)$ admits local sections which are smooth in the aforementioned sense.
Here $S$ is any manifold diffeomorphic to $N$.

Let $\iota_N\colon N\to M$ denote the natural inclusion of $N$ and write
\[
	G_{\iota_N}
	=\bigl\{g\in G\mid g\circ\iota_N=\iota_N\bigr\}
	=\bigl\{g\in G\mid\forall x\in N:g\cdot x=x\bigr\}
\]
for the isotropy group at $\iota_N$, i.e., the kernel of the homomorphism in~\eqref{E:homo.sigma}.
We say \emph{$G$ has good isotropy at $\iota_N$} if $G_{\iota_N}$ is a splitting Lie subgroup of $G$.
We say \emph{$G_N$ has good isotropy at $\iota_N$} if the inclusions $G_{\iota_N}\subseteq G_N\subseteq G$ are both splitting Lie subgroups.
In the latter case $G$ has good isotropy at $\iota_N$ also.

Proceeding exactly as in the proof of Lemma~\ref{L:geom1} one shows:

\begin{lemma}\label{L:geom2}
	Suppose $G$ has good extension at $N$ and suppose $G_N$ has good isotropy at $\io_N$.
	Then the orbits of the $G_N$ action on $\Diff(S,N)$ are open and closed.
	Moreover, for $\varphi\in\Diff(S,N)$ the orbit map $G_N\to G_N\cdot\varphi$, $g\mapsto g\circ\varphi$ is a smooth principal bundle with structure group $G_{\iota_N}$.

	In particular, the map in~\eqref{E:homo.sigma} is a smooth principal $G_{\iota_N}$ bundle, after disregarding the connected components of $\Diff(N)$ which are not in its image.
\end{lemma}

In the remaining part of this section we will exemplify these notions by means of several classical groups of diffeomorphisms.


\subsection{The full diffeomorphism group}\label{S:geom.plain}
The group of compactly supported diffeomorphisms $G=\Diff_c(M)$ is a regular Lie group in the convenient setting of \cite{KM97}.
Its Lie algebra is the Lie algebra of compactly supported vector fields, $\mathfrak g=\X_c(M)$.

The natural $\Diff_c(M)$ action on $\Gr(M)$ admits local smooth sections and the isotropy group $\Diff_c(M;N):=\{f\in\Diff_c(M):f(N)=N\}$ is a splitting Lie subgroup in $\Diff_c(M)$ with Lie algebra $\mathfrak X_c(M;N)=\{X\in\mathfrak X_c(M):X(N)\subseteq TN\}$.
Hence, the orbit $\Diff_c(M)\cdot N$ consists of several connected components of $\Gr(M)$ and the orbit map $\Diff_c(M)\to\Diff_c(M)\cdot N$, $f\mapsto f(N)$ is a smooth principal bundle with structure group $\Diff_c(M;N)$, by Lemma~\ref{L:geom1}.
Moreover, $\Diff_c(M)$ has good extension at $N$ and $\Diff_c(M;\iota_N)=\{f\in\Diff_c(M):f|_N=\id\}$ is a splitting Lie subgroup of $\Diff_c(M;N)$ with Lie algebra $\mathfrak X_c(M;\iota_N)=\{X\in\mathfrak X_c(M):X|_N=0\}$.
Hence, the orbit $\Diff_c(M;N)\cdot\varphi$ consists of several connected components of $\Diff(S,N)$ and the orbit map $\Diff_c(M;N)\to\Diff_c(M;N)\cdot\varphi$, $g\mapsto g\circ\varphi$ is a smooth principal bundle with structure group $\Diff_c(M;\iota_N)$, according to Lemma~\ref{L:geom2}.
These facts are well known, see \cite[Lemma~1.2]{HV20} for instance.
We summarize them in

\begin{proposition}[{\cite[Lemma~1.2]{HV20}}]\label{P:DiffcM.good}
	Suppose $G=\Diff_c(M)$ and $N\in\Gr(M)$.
	Then $G$ has good orbit, good extension, and good isotropy at $N$, and $G_N$ has good isotropy at $\iota_N$.
	Moreover, the orbit $G\cdot N$ is open and closed in $\Gr(M)$.
\end{proposition}


\subsection{The group of principal bundle automorphisms}\label{S:geom.pb}
Let $P\to M$ denote a finite dimensional principal $H$ bundle. Then the group $G=\Aut_c(P)$ of $H$ equivariant diffeomorphisms of $P$ with compact support in $M$	is a convenient Lie group in a canonical way.
The canonical homomorphism $G\to\Diff_c(M)$ is a smooth principal bundle \cite[Theorem~3.1, Corollary~4.4]{AM}, after disregarding the connected components of $\Diff_c(M)$ which are not in its image.
Hence, Proposition~\ref{P:DiffcM.good} implies:

\begin{proposition}
	Suppose $G=\Aut_c(P)$ and $N\in\Gr(M)$.
	Then $G$ has good orbit, good extension, and good isotropy at $N$, and $G_N$ has good isotropy at $\iota_N$.		
	Moreover, the orbit $G\cdot N$ is open and closed in $\Gr(M)$.
\end{proposition}


\subsection{The Hamiltonian diffeomorphism group}\label{S:isosymp}	
Throughout this section $(M,\om)$ denotes a symplectic manifold.
The Lie algebra of compactly supported Hamiltonian vector fields will be denoted by $\ham_c(M)=\{X_h:h\in C^\infty_c(M)\}$, where $i_{X_h}\om=dh$.
This is an ideal in the Lie algebra $\mathfrak X_c(M,\om)$ of compactly supported symplectic vector fields.
The Lie algebra $\ham_c(M)$ can be identified with the Lie algebra $C^\oo_0(M)$ of all compactly supported functions on $M$ 
for which the integral with respect to the Liouville form vanishes on all closed connected components,
endowed with the Poisson bracket $\{h,h'\}=\om(X_{h'},X_h)$.

Let  $\Ham_c(M)$ denote the group of compactly supported Hamiltonian diffeomorphisms,
i.e., the group of diffeomorphisms 
obtained by integrating time dependent vector fields in $\ham_c(M)$.
This is a normal subgroup in the group $\Diff_c(M,\om)$ of compactly supported symplectic diffeomorphisms.
The group $\Ham_c(M)$ is a convenient Lie group in a natural way \cite[Theorem~43.13 and Remark~43.14]{KM97}.

The regular Lie group $\Diff_c(M,\om)$ of symplectic diffeomorphisms \cite[Theorem~43.12]{KM97} will be treated elsewhere.

\subsubsection*{Isotropic submanifold} We first consider the case of isotropic submanifolds.
In view of the tubular neighborhood theorem for isotropic embeddings \cite{W77,W81}, 
the space $\Gr^{\iso}(M,\om)$ of all isotropic submanifolds in $M$ is a splitting smooth submanifold of $\Gr(M)$, 
see for instance \cite[Section~8]{L09}, with tangent space 
\begin{equation*}
	T_N\Gr^{\iso}(M,\om)
	=\bigl\{X\in\Ga(TM|_N/TN): \io_N^*i_{X}\om\in\Om^1(N) \text{ closed}\bigr\}.
\end{equation*}
Weinstein's \cite{W90} \emph{isodrastic distribution} on $\Gr^{\iso}(M,\om)$ is given by
\begin{equation}\label{DDDDn}
	\D_N:=\{X\in\Ga(TM|_N/TN): \io_N^*i_{X}\om\in dC^\oo(N)\}
\end{equation}
and has finite codimension $\dim H^1(N;\RR)$.
This is an integrable distribution, whose leaves are orbits of the group of Hamiltonian diffeomorphisms $\Ham_c(M)$ \cite{W90,L09}.
In particular, each $\Ham_c(M)$ orbit through an isotropic $N\in\Gr^{\iso}(M,\om)$ is a splitting initial smooth submanifold in $\Gr(M)$.
Furthermore, the $\Ham_c(M)$ action on each orbit admits local smooth sections \cite[Corollary~3.2]{HV23}.
Hence we have good orbits through isotropic $N$.
From \cite[Proposition~3.1]{HV23} follows also the good extension property at isotropic $N$.
We summarize this in the following:

\begin{proposition}\label{pp}
	If $G=\Ham_c(M)$, then at any $N\in \Gr^{\iso}(M,\om)$ we have good orbit and good extension.
	Moreover, $G_N$ has good (open) orbits at any $\ph\in\Diff(S,N)$.
\end{proposition}

We do not know if $\Ham_c(M)$ has good isotropy at isotropic $N$ or $\iota_N$.

\subsubsection*{Symplectic submanifold} 
We now consider the symplectic  Grass\-man\-nian $\Gr^{\symp}(M,\om)$ consisting of all closed symplectic submanifolds of $M$.
Clearly, this is a open subset in $\Gr(M)$ which is invariant under $\Diff(M,\omega)$.

Combining the construction in the proof of \cite[Proposition 3]{HV04} with Lemma~\ref{L:loc.sec} we obtain:

\begin{proposition}\label{P:p}
	The $\Ham_c(M)$ action on $\Gr^{\symp}(M,\om)$ admits local smooth sections.
	In particular, $\Ham_c(M)$ has good (open) orbit at every $N\in\Gr^{\symp}(M,\om)$.
\end{proposition}

The map
\begin{equation}\label{E:J.symp}
	J\colon\Gr^{\symp}(M,\om) \to \ham_c(M)^*=C^\oo_0(M)^*,\quad\langle J(N),X_h\rangle = \int_{N}he^\om,
\end{equation}
is injective and $\Diff(M,\omega)$ equivariant. 
Now, for any symplectic submanifold $N\subseteq M$, the coadjoint orbit of $J(N)$ is a good coadjoint orbit of $\Ham_c(M)$.
More precisely, by applying Lemma~\ref{lemac} we obtain:

\begin{theorem}[{\cite[Theorem~3]{HV04}}]\label{P:symp}
	Let $(M,\omega)$ be a symplectic manifold and let $\M$ denote a connected component of the symplectic Grassmannian $\Gr^{\symp}(M,\om)$.
	Then its image $J(\M)$ under the map $J$ in~\eqref{E:J.symp} is a good coadjoint orbit of $\Ham_c(M)$ and the (formal) pullback of the KKS symplectic form, 
	denoted  by $\Omega=J^*\om_{\KKS}$, is a $\Ham_c(M)$ equivariant smooth (weakly) symplectic form on $\M$ characterized by
	\[
		\Omega_N(X,Y)=\int_N\io_N^*(i_Yi_X e^\om), 
	\]
	for $X,Y\in\Ga(TM|_N/TN)=T_N\Gr(M)=T_N\Gr^{\symp}(M,\om)$.
	Furthermore, $J$ is an (equivariant) moment map for the (Hamiltonian) action of $\Ham_c(M)$ on $\M$.		
\end{theorem}


\subsection{Group of exact volume preserving diffeomorphisms}\label{S:isovol}
Throughout this section $(M,\mu)$ denotes a manifold endowed with a volume form.
We will denote the Lie algebra of compactly supported exact divergence free vector fields by 
\[
	\mathfrak X_{c,\ex}(M):=\bigl\{X\in\mathfrak X_c(M):i_X\mu=d\alpha,\ \alpha\text{ with compact support}\bigr\},
\]
an ideal in the Lie algebra $\mathfrak X_c(M,\mu)$ of compactly supported divergence free vector fields.
Correspondingly, we let $\Diff_{c,\ex}(M)$ denote the group of compactly supported exact volume preserving diffeomorphisms, 
obtained by integrating time dependent vector fields in $\mathfrak X_{c,\ex}(M)$.
This is a normal subgroup in the group $\Diff_c(M,\mu)$ of compactly supported volume preserving diffeomorphisms.

\begin{remark}\label{R:enlarge}
For $M$  compact, the groups  $\Diff(M,\mu)$ and $\Diff_{\ex}(M)$ are regular Lie groups \cite[Theorem~III.2.5.3]{Ham}.
Though this fact is not (yet) known for noncompact $M$, one can handle this case too.
The definition of an action admitting local smooth section  can be enlarged to any subgroup $G$ of a regular Lie group $H$,
by requiring the local section $u:U\to G$ to be smooth as an $H$ valued map.
We notice that, in this more general setting, Lemma \ref{L:loc.sec} still holds.
We apply this to the subgroups $\Diff_{c}(M,\mu)$ and $\Diff_{c,\ex}(M)$ of the regular Lie group $\Diff_c(M)$.
\end{remark}

\begin{proposition}\label{P:Diffmu.good.ext}
	$\Diff_{c,\ex}(M)$ has good extension at each $N\in\Gr_{\codim\geq1}(M)$.
\end{proposition}

\begin{proof}
	Suppose $N$ is a closed submanifold of positive codimension in $M$.
	It is well known that every vector field on $N$ can be extended to an exact volume preserving vector field on $M$.
	More precisely, there exists a bounded linear map $s\colon\mathfrak X(N)\to\mathfrak X_{c,\ex}(M)$ such that $s(X)|_N=X$ for all $X\in\mathfrak X(N)$.
	Let us briefly recall the construction of such a map $s$.
	To this end suppose $X\in\mathfrak X(N)$ and extend it to a vector field $\tilde X$ on an open neighborhood of $N$, i.e., $\tilde X|_N=X$.
	Using the relative Poincar\'e lemma, we obtain a form $\lambda$ such that $\lambda|_N=0$ and $d\lambda=di_{\tilde X}\mu$ in a neighborhood of $N$.
	Let $Y$ denote the vector field such that $i_Y\mu=\lambda$ and put $Z=\tilde X-Y$.
	Then $Z|_N=X$ and $di_Z\mu=0$ in a neighborhood of $N$.
	As $Z$ is tangential to $N$, the closed form $i_Z\mu$ vanishes when pulled back to $N$.
	Hence, by the relative Poincar\'e lemma, there exists a form $\rho$ such that $d\rho=i_Z\mu$ in a neighborhood of $N$.
	Multiplying $\rho$ with a compactly supported bump function, we obtain a globally defined compactly supported form $\alpha$ on $M$ such that $d\alpha=i_Z\mu$ in a neighborhood of $N$.
	Hence, defining $s(X)\in\mathfrak X_{c,\ex}(M)$ by $i_{s(X)}\mu=d\alpha$ we have $s(X)|_N=X$.
	Now let $\theta\in\Omega^1(\Diff(N);\mathfrak X(N))$ denote the Maurer--Cartan form of $\Diff(N)$ and apply Lemma~\ref{L:loc.sec} with $\sigma=s\circ\theta$ to obtain a local right inverse for the restriction map $\Diff_{c,\ex}(M,N)\to\Diff(N)$.
\end{proof}

\subsubsection*{Codimension $\ge 2$}
The action of $G=\Diff_{c,\ex}(M)$ on connected components of the nonlinear Grassmannian $\Gr_{\codim \ge 2}(M)$ of closed submanifolds of $M$ of codimension at least two 
is known to be transitive \cite[Proposition~2]{HV04}.
We require the following slightly stronger result.

\begin{proposition}\label{trans}
	The $\Diff_{c,\ex}(M)$ action on $\Gr_{\codim \ge 2}(M)$ admits local smooth sections.
	In particular, $\Diff_{c,\ex}(M)$ has good (open) orbit at each $N\in\Gr_{\codim\ge 2}(M)$.
\end{proposition}

\begin{proof}
	Let $\pi\colon E\to N$ denote the vector bundle projection of the normal bundle, $E:=TM|_N/TN$. 
	For simplicity, we assume $N$ and $E$ orientable.
	Let $\nu$ be a volume form on $N$ and let $\Omega$ be a nowhere vanishing section of $\Lambda^{\operatorname{rk}(E)}E^*$.
	We consider $\Omega$ as a family of volume forms on the fibers of $E$ that are translation invariant, i.e., constant.
	Then $\tilde\mu=\Omega\wedge\pi^*\nu$ is a volume form on the total space of $E$. (This does not require a connection.)
	For $X\in\Gamma(E)$, we let $\tau_X$ denote the diffeomorphism given by fiberwise translation with $X$.
	More explicitly, $\tau_X\colon E\to E$, $\tau_X(e)=e+X(\pi(e))$ for $e\in E$.
	By construction, $\tau_X(N)=X(N)$.
	According to the Cavalieri principle, $\tau_X^*\tilde\mu=\tilde\mu$.
	Moreover, $\tau_{tX}$ is a 1-parameter group of diffeomorphisms with generating vector field 
	$\frac\partial{\partial t}|_{t=0}\tau_{tX}=X$, 
	now considered as a vertical vector field on the total space of $E$ which is fiberwise constant.
	Since the rank of $E$ is at least two, $i_X\tilde\mu=d\alpha_X$ where $\alpha_X=\frac1{\operatorname{rk}(E)-1}i_Ri_X\Omega\wedge\pi^*\nu$ 
	and $R$ denotes the fiberwise radial (Euler) vector field on $E$, that is, $R(e)=\frac\partial{\partial t}|_{t=1}te$ for $e\in E$.
	
	Choose a compactly supported function $\lambda$ on the total space of $E$ which is 1 on an open, 
	fiberwise radial neighborhood $U$ of the zero section in $E$.
	Define $Y_X\in\mathfrak X_{c,\ex}(E,\tilde\mu)$ by $i_{Y_X}\tilde\mu=d(\lambda\alpha_X)$ 
	and let $g_X\in\Diff_{c,\ex}(E,\tilde\mu)$ denote its flow at time one.
	Since $Y_X=X$ over $U$, and since $U$ is radial, we obtain
	\[
		X(N)=g_X(N)
	\]
	whenever $X(N)\subseteq U$.
	Put $\mathcal U=\{X\in\Gamma(E):X(N)\subseteq U\}$ and recall that assigning to $X\in\mathcal U$ the submanifold $X(N)$ in $E$ provides a standard chart for the smooth structure of $\Gr(E)$ around $N$.
	Hence, up to this chart, $\mathcal U\to\Diff_{c,\ex}(E,\tilde\mu)$, $X\mapsto g_X$ is a smooth right inverse for the orbit map $\Diff_{c,\ex}(E,\tilde\mu)\to\Gr(E)$, $g\mapsto g(N)$.
	This settles the local problem.	
	To reduce the global problem to the local one we use Moser's lemma \cite{M65} to find a tubular neighborhood $E\to M$ of $N$ such that, in a neighborhood of the zero section, the volume form $\mu$ is of the form $\tilde\mu$ considered above.
\end{proof}

We consider the oriented nonlinear Grassmannian $\Gr^{\rm or}(M)$ of Example~\ref{Ex:deco.or}.
Since the covering $\Gr^{\rm or}(M)\to\Gr(M)$ admits local sections, Proposition~\ref{trans} implies that the $\Diff_{c,\ex}(M)$ action on $\Gr^{\rm or}_{\codim\geq2}(M)$ admits local smooth sections too.
Let $\Gr^{\rm or,\hnull}_{\codim=2}(M)$ denote the union of all connected components consisting of null-homologous submanifolds of codimension two.
Clearly, $\Gr^{\rm or,\hnull}_{\codim=2}(M)$ is stable under $\Diff(M,\mu)$.
We have an injective $\Diff(M,\mu)$ equivariant map
\begin{equation}\label{E:J.or.hnull}
	J\colon\Gr^{\rm or,\hnull}_{\codim=2}(M)\to \mathfrak X_{c,\ex}(M)^*,\quad \langle J(N),X\rangle=\int_{N}\io_{N}^*\al,
\end{equation}
where $\al$ is any (compactly supported) potential form for $X\in\mathfrak X_{c,\ex}(M)$. 

Now, for any $N\in\Gr^{\rm or,\hnull}_{\codim=2}(M)$ the coadjoint orbit of $J(N)$ is a good coadjoint orbit of $\Diff_{c,\ex}(M)$.
More precisely, by applying Lemma~\ref{lemac} we obtain:

\begin{theorem}[{\cite[Theorem~2]{HV04}}]\label{P:codim2}
	Let $\mu$ be a volume form on $M$ and consider a connected component $\M$ in the augmented Grassmannian $\Gr^{\rm or,hnull}_{\codim=2}(M)$.
	Then its image $J(\M)$ under the map $J$ in~\eqref{E:J.or.hnull} is a good coadjoint orbit of $\Diff_{c,\ex}(M)$ and the (formal) pullback of the KKS symplectic form, 
	denoted  by $\omega=J^*\om_{\KKS}$, is a $\Diff_{c,\ex}(M)$ equivariant smooth (weakly) symplectic form on $\M$, 
	namely the Marsden--Weinstein \cite{MW83} symplectic form 
	\[
		\om_N(X,Y)=\int_N\io_N^*(i_Yi_X\mu),
	\]
	where $X,Y\in\Gamma(TM|_N/TN)=T_N\Gr^{\rm or,\hnull}_{\codim=2}(M)$.
	Furthermore, $J$ is an (equivariant) moment map for the (Hamiltonian) action of $\Diff_{c,\ex}(M)$ on $\M$.		
\end{theorem}

The remaining connected components of $\Gr^{\rm or}_{\codim=2}(M)$ for compact $M$ have been also considered in \cite{HV04}.
Assuming that the total volume $\int_M\mu$ is an integer, they parametrize coadjoint orbits of central extensions of $\Diff_{c,\ex}(M)$ (see also \cite{ismagilov}).

\subsubsection*{Codimension one}
In this situation, the volume form gives rise to an integrable distribution on the codimension one nonlinear Grassmannian $\Gr_{\codim=1}(M)$,
which are analogous to Weinstein's isodrastic foliation in the symplectic case \cite{W90,L09} explained in Section~\ref{S:isosymp}.

The \emph{isodrastic distribution} on $\Gr_{\codim=1}(M)$, introduced in \cite{HV24}, is defined  by
\begin{equation}\label{E:dn}
	\D_N:=\bigl\{X\in\Ga(TM|_N/TN):{\io_N^*(i_{X}\mu)\in\Om^{\dim N}(N)\text{ is exact}}\bigr\}.
\end{equation}
It has finite codimension equal to $\dim H^0(N;\mathfrak o_N)$, the number of orientable connected components of $N$.
This can be seen using the isomorphism of line bundles 
\begin{equation}\label{E:mu}
	\mu_N\colon TM|_N/TN\to\Lambda^{\dim N}T^*N,
\end{equation}
given by contraction with the volume form followed by restriction to $TN$.

Note that the null-homologous isodrastic leaves in $\Gr_{\codim=1}(M)$ are characterized by fixed enclosed volume.
Another characterization works for open manifolds $M$:
the isodrasts in $\Gr_S(M)$ are the connected components of the level sets of the map $F(N)=\int_N\vartheta$, for a fixed $\vartheta$ with $\mu=d\vartheta$.

\begin{proposition}\label{L:Fcex.int}
	The isodrastic distribution $\mathcal D$ on $\Gr_{\codim=1}(M)$ is integrable:
	its leaves coincide with the orbits of $\Diff_{c,\ex}(M)$. 
	Moreover, the action on each isodrastic leaf admits local smooth sections.		
	In particular, the group $\Diff_{c,\ex}(M)$ has good orbit at any submanifold $N$ of codimension one.
\end{proposition}

\begin{proof}
	The first statement is \cite[Lemma~4.2]{HV24}. 
	For the second statement, we fix a submanifold $N\in\Gr_{\codim=1}(M)$, assumed w.l.o.g.~to be connected and orientable, 
	and we denote by $\L$ the isodrastic leaf through $N$.
	Subsequently, we use Lemma~\ref{L:loc.sec} to show that the $\Diff_{c,\ex}(M)$ action on the leaf $\L$ admits local smooth sections.

	We consider a tubular neighborhood of $N$, identified with $N\x\R$, where the volume form can be expressed as 
	$\mu=dt\wedge p^*\nu$ with $\nu$ a volume form on $N$. Here $p$ and $t$ are the projections to $N$ and $\RR$.
	The graph of a function $f\in C^\oo(N)$ is a submanifold $N_f\subseteq M$. This provides a chart  $U\subseteq\Gr(M)$ centered at $N$
	and the level sets of the smooth function $V(N_f):=\int_Nf\nu$ are leaves of $\mathcal D|_U$.

	In particular, the connected component of $N$ in $\L\cap U$, an open neighborhood  of $N$ in $\L$, is the zero set of the function $V$.
	The chart above is a submanifold chart: 
	$\L\cap U$ is parametrized by the subspace  of zero integral functions $C_0^\oo(N)=\{f\in C^\oo(N):\int_Nf\nu=0\}$.
	In this chart a tangent vector at $N_f$ is given by a pair $(f,h)$ with $f,h\in C_0^\oo(N;\RR)$.
	There exists a smooth right inverse $\phi$ to $d:\Om^{n-1}(N)\to d\Om^{n-1}(N)$, thus to each $h\in C_0^\oo(N;\RR)$
	one associates in a smooth way a differential form  $\ga:=\phi(h\nu)$ with the property for  $h\nu=d\ga$.

	For $a>0$, let $U_a\subseteq U$ be the open subset consisting of those submanifolds $N_f\in U$ with $\max_{x\in N}|f(x)|<a$, 
	and let  $b\in C^\infty_c(\mathbb R)$  such that $b=1$ on the interval $[-a,a]$.
	We define a smooth map $\si:T\L|_{\L\cap U_a}\to\X_{c,\ex}(M)$ by $\si(f,h)=X_\al$, 
	where the potential form $\al$ is given by $(t^*b)(p^*\ga)$ on the tubular neighborhood of $N$, and extended by zero to $M$.
	Since by construction $\ze_X(N_f)=(f,h)$, Lemma \ref{L:loc.sec} can be applied to get that the action on $\L$ admits local smooth sections.
\end{proof}


\subsection{Group of volume preserving diffeomorphisms}\label{S:2.7}
Let $G=\Diff_{c}(M,\mu)$.
For the nonlinear Grassmannian $\Gr_{\codim \ge 2}(M)$ of codimension at least two, Proposition~\ref{trans} holds for this group too.
For the codimension one nonlinear Grassmannian $\Gr_{\codim=1}(M)$, we consider a variation of the isodrastic distribution \eqref{E:dn}:
\begin{equation}\label{E:dnex}
	\D_{N,\mu}:=\{X\in\Ga(TM|_N/TN):[i_{X}\mu]\in \io_N^*H_c^{\dim N}(M;\mathbb R)\},
\end{equation}
for the natural homomorphism
\begin{equation}\label{E:iotaN*}
	\iota_N^*\colon H_c^{\dim N}(M;\mathbb R)\to H^{\dim N}(N;\mathbb R).
\end{equation}
In view of Proposition~\ref{P:below} below, we will call $\mathcal D_\mu$ the \emph{isovolume distribution}.
In general, the codimension of $\mathcal D_\mu$, which is at most the codimension of $\mathcal D$, 
will depend on the connected component of $\Gr_{\codim=1}(M)$:
\begin{equation}\label{E:codim.Fcmu}
	\codim\mathcal D_{N,\mu}=\dim\coker\bigl(H_c^{\dim N}(M;\mathbb R)\to H^{\dim N}(N;\mathbb R)\bigr).
\end{equation}
If every orientable connected component of $N$ is null-homologous in $M$, then the distributions $\mathcal D_\mu$ and $\mathcal D$ coincide at $N$.
The converse is not true in general: at a meridian $N=S^1\times\{y\}$ in the cylinder $M=S^1\times\mathbb R$ the two distributions coincide,
still the homology class of $N$ in $M$ is not zero.
	
\begin{proposition}[{\cite[Proposition~4.9]{HV24}}]\label{P:below}
	Let $N_t\in\Gr_{\codim =1}(M)$ be a smooth curve, $t\in[0,1]$.
	
	(a) If $N_t$ is tangent to the isovolume distribution $\mathcal D_\mu$, then there exists a compact submanifold with boundary, $K\subseteq M$, 
	whose interior contains $N_t$ for all $t$, and such that the volume of each connected component of $K\setminus N_t$ is constant in $t$.
	
	(b) Suppose $K\subseteq M$ is a compact submanifold with boundary, whose interior contains $N_t$ for all $t$.
	If, moreover, the volume of each connected component of $K\setminus N_t$ is independent of $t$, then $N_t$ is tangent to the isovolume distribution $\mathcal D_\mu$.
\end{proposition}

Clearly, the isovolume distribution is invariant under the action of $\Diff(M,\mu)$.
Using \cite[Lemma~4.7]{HV24},  we get the following analogue of Proposition~\ref{L:Fcex.int}:
	
\begin{proposition}\label{L:cmu}
	The isovolume distribution $\mathcal D_\mu$ on $\Gr_{\codim=1}(M)$ is integrable.
	Its leaves coincide with the orbits of $\Diff_{c}(M,\mu)_0$, the identity component of the group of volume preserving diffeomorphisms. 
	Moreover, the action on each isodrastic leaf admits local smooth sections.
	In particular, the group $G=\Diff_{c}(M,\mu)$ has good orbit at any $N$ of codimension one.
\end{proposition}
	
In a future work we will study, in the symplectic setting, an analogous distribution on the isotropic nonlinear Grassmannian,
called the isosymplectic distribution. Its leaves coincide with the orbits of the identity component of the group of symplectic diffeomorphisms.


\subsection{Group of contact diffeomorphisms}\label{Ex:geom.contact}
Let $\xi$ denote a contact distribution on $M$ and let $G=\Diff_c(M,\xi)$ denote the group of compactly supported contact diffeomorphisms.
This is a convenient Lie group in a natural way \cite[Theorem~43.19]{KM97}.
Recall \cite[Lemma~4.8]{HV22} that the space of contact isotropic submanifolds, $\Gr^{\iso}(M,\xi)$, 
is a splitting smooth submanifold in $\Gr(M)$ that is invariant under $\Diff_c(M,\xi)$.

\begin{proposition}\label{P:contact.good}
	Suppose $G=\Diff_c(M,\xi)$ and $N\in\Gr^{\iso}(M,\xi)$.
	Then $G$ has good orbit and good extension at $N$.
	Moreover, the orbit $G\cdot N$ is open and closed in $\Gr^{\iso}(M,\xi)$.
	Furthermore, $G_N$ has good (open) orbit at any $\ph\in\Diff(S,N)$. 
\end{proposition}

\begin{proof}
	Recall that $\Gr^{\iso}(M,\xi)$ is a splitting smooth submanifold in $\Gr(M)$.
	From \cite[Theorem~3.5]{HV22} we conclude that the $G$ action on $\Gr^{\iso}(M,\xi)$ admits local smooth sections.
	Hence, $G$ has good orbit at $N$ and $G\cdot N$ is open and closed in $\Gr^{\iso}(M,\xi)$.

	To prove good extension at contact isotropic $N$, we first observe that vector fields on $N$ can be extended to contact vector fields on $M$.
	More precisely, there exists a bounded linear map $s\colon\mathfrak X(N)\to\mathfrak X_c(M;\xi)$ 
	such that $s(Z)|_N=Z$ for all $Z\in\mathfrak X(N)$.
	To construct such an $s$ we assume, for simplicity, that $\xi$ is cooriented, i.e., $\xi=\ker\alpha$ for a contact 1-form $\alpha$.
	Suppose $Z\in\mathfrak X(N)$.
	As $N$ is isotropic, we have $\alpha(Z)=0$ and $d\alpha(Z,Y)=0$ for all $Y\in\mathfrak X(N)$.
	We consider $i_Zd\alpha\in\Gamma((TM|_N/TN)^*)$ as a fiber wise linear function on the normal bundle $TM|_N/TN$.
	Using a tubular neighborhood of $N$ in $M$ and multiplying by a compactly supported function which is $-1$ in a neighborhood of the zero section, we obtain $h\in C^\infty_c(M,\mathbb R)$ such that $h|_N=0$ and $dh(Y)+d\alpha(Z,Y)=0$ for all $Y\in\Gamma(TM|_N)$.
	Let $s(Z)=X_h\in\mathfrak X_c(M,\xi)$ denote the unique contact vector field such that $\alpha(X_h)=h$.
	By construction, $X_h|_N=Z$.
	Now let $\theta\in\Omega^1(\Diff(N);\mathfrak X(N))$ denote the Maurer--Cartan form of $\Diff(N)$ and apply Lemma~\ref{L:loc.sec} with $\sigma=s\circ\theta$ to obtain a local right inverse for $G_N\to\Diff(N)$.

	The remaining statement follows immediately.
\end{proof}

We do not know if $\Diff_c(M,\xi)$ has good isotropy at isotropic $N$ or $\iota_N$.

\section{Decorated nonlinear Grassmannians}\label{S:deco}

Let $S$ be a closed manifold, allowed to be nonconnected and nonorientable.
For each manifold $M$, we let $\Emb(S,M)$ denote the space of all parametrized submanifolds of type $S$ in $M$, i.e., the space of all smooth embeddings of $S$ into $M$.
It is an open subset in the Fr\'echet manifold $C^\infty(S,M)$, thus  a smooth Fr\'echet manifold.
We let $\Gr_S(M)$ denote the \emph{nonlinear Grassmannian of type $S$ in $M$,} i.e., the space of all smooth submanifolds in $M$ that are diffeomorphic to $S$. It is open and closed in the nonlinear Grassmannian $\Gr(M)$ of all closed submanifolds in $M$.

The nonlinear Grassmannian $\Gr_S(M)$ is the base of a (locally trivial) smooth principal bundle
with structure group Diff(S) (via the $\Diff(M)$ equivariant map)
\begin{equation}\label{emb}
	\pi:\Emb(S,M)\to\Gr_S(M),\qquad\pi(\varphi)=\varphi(S),
\end{equation}
see \cite{BF81,M80b,M80c,GBV14} and \cite[Theorem~44.1]{KM97}. 
The decorated nonlinear Grassmannians described in this section can be seen as associated bundles to this principal bundle.


\subsection{Framework for decorations on nonlinear Grassmannians}\label{SS:deco}
Let $\catclosed$ denote the category of (finite dimensional) closed manifolds and diffeomorphisms.
A \emph{decoration functor} is a functor $\mathfrak D$ from the category $\catclosed$ into the category of Fr\'echet manifolds and smooth maps such that the $\Diff(S)$ action on $\mathfrak D(S)$ is smooth, for each closed manifold $S$.
For a finite dimensional manifold $M$, define the space of \emph{$\mathfrak D$-decorated Grassmannians of type $S$ in $M$} by
\[
	\Gr^{\deco}_S(M):=\bigl\{(N,\nu):N\in\Gr_S(M),\nu\in\mathfrak D(N)\bigr\}.
\]
Note that $\Diff(M)$ acts naturally on $\Gr^{\deco}_S(M)$.
More explicitly, if $g\in\Diff(M)$ and $(N,\nu)\in\Gr^{\deco}_S(M)$, then $g\cdot(N,\nu)=(g(N),\mathfrak D(g|_N)\cdot\nu)$, where $g|_N\colon N\to g(N)$ denotes the restricted diffeomorphism and $\mathfrak D(g|_N)\colon\mathfrak D(N)\to\mathfrak D(g(N))$.

We use the canonical $\Diff(M)$ equivariant identification
\begin{equation}\label{E:abcd}
	\Gr^{\deco}_S(M)=\Emb(S,M)\times_{\Diff(S)}\mathfrak D(S),\quad\bigl(\varphi(S),\mathfrak D(\varphi)\cdot\mu\bigr)\leftrightarrow[\varphi,\mu].
\end{equation}
to equip $\Gr^{\deco}_S(M)$ with a smooth structure.
Then the $\Diff(M)$ equivariant forgetful map 
\begin{equation}\label{E:fb.deco}
	\Gr^{\deco}_S(M)\to\Gr_S(M),\quad(N,\nu)\mapsto N
\end{equation} 
is a smooth fiber bundle with typical fiber $\mathfrak D(S)$, and the $\Diff_c(M)$ action on $\Gr^{\deco}_S(M)$ is smooth.
As $\Gr_S(M)$ is open and closed in $\Gr(M)$, this also yields a smooth structure on the disjoint union $\Gr^{\deco}(M)=\bigsqcup_S\Gr^{\deco}_S(M)$ of all $\mathfrak D$-decorated Grassmannians
such that  $\Gr^{\deco}(M)\to\Gr(M)$ is a smooth fiber bundle with fiber $\mathfrak D(N)$ over $N\in\Gr(M)$.

Suppose $\mu\in\mathfrak D(S)$.
Define the space of \emph{$\mathfrak D$-decorated Grassmannians of type $(S,\mu)$ in $M$} by
\begin{equation}\label{esmu}
	\Gr^{\deco}_{S,\mu}(M):=\bigl\{(N,\nu)\in\Gr^{\deco}_S(M):(N,\nu)\cong(S,\mu)\bigr\}.
\end{equation}
Here $(N,\nu)\cong(S,\mu)$ is short for: there exists a diffeomorphism $\varphi\colon S\to N$ such that $\mathfrak D(\varphi)\cdot\mu=\nu$, where $\mathfrak D(\varphi)\colon\mathfrak D(S)\to\mathfrak D(N)$.
Note that $\Gr^{\deco}_{S,\mu}(M)$ is invariant under the action of $\Diff(M)$.

Let $\Diff(S)\cdot\mu$ denote the $\Diff(S)$ orbit of $\mu$.
We say \emph{$\Diff(S)$ has good orbit at $\mu$} if $\Diff(S)\cdot\mu$ is an initial splitting smooth submanifold in $\mathfrak D(S)$ and the (smooth) $\Diff(S)$ action on $\Diff(S)\cdot\mu$ admits local smooth sections.

\begin{lemma}\label{L:deco1}
	Suppose $\Diff(S)$ has good orbit at $\mu$.
	Then $\Gr^{\deco}_{S,\mu}(M)$ is an initial splitting smooth submanifold in $\Gr^{\deco}_S(M)$ and the (smooth) $\Diff_c(M)$ action on $\Gr^{\deco}_{S,\mu}(M)$ admits local smooth sections.
	Moreover, the $\Diff(M)$ equivariant forgetful map obtained by restricting \eqref{E:fb.deco},
	\begin{equation}\label{E:fb.deco.mu}
		\Gr^{\deco}_{S,\mu}(M)\to\Gr_S(M),
	\end{equation}
	is a smooth fiber bundle with typical fiber $\Diff(S)\cdot\mu$.
	Furthermore, the bijection in~\eqref{E:abcd} restricts to a canonical $\Diff(M)$ equivariant diffeomorphism
	\begin{equation}\label{E:fb.deco.ass}
		\Gr^{\deco}_{S,\mu}(M)=\Emb(S,M)\times_{\Diff(S)}\Diff(S)\cdot\mu.
	\end{equation}
\end{lemma}

\begin{proof}
	Most of the statements follow immediately from the fact that the identification in \eqref{E:fb.deco.ass} 
	is the restriction of the diffeomorphism in \eqref{E:abcd}.
	Since the $\Diff_c(M)$ action on $\Emb(S,M)$ admits local smooth sections \cite[Lemma~2.1(c)]{HV20} we can use 
	Lemma \ref{before.sec} to conclude that the $\Diff_c(M)$ action on $\Gr^{\deco}_{S,\mu}(M)$ admits local smooth sections.
\end{proof}

We say \emph{$\Diff(S)$ has good isotropy at $\mu\in\DD(S)$} if $\Diff(S,\mu)=\{f\in\Diff(S):\mathfrak D(f)\cdot\mu=\mu\}$ is a splitting Lie subgroup in $\Diff(S)$.

\begin{lemma}\label{L:deco2}
	Suppose $\Diff(S)$ has good orbit and good isotropy at $\mu\in\mathfrak D(S)$.
	Then the $\Diff(S)$ equivariant orbit map
	\begin{equation}\label{E:Diff.S.orbit}
		\Diff(S)\to\Diff(S)\cdot\mu,\quad f\mapsto\mathfrak D(f)\cdot\mu,
	\end{equation}
	is a smooth principal $\Diff(S,\mu)$ bundle, and so is the $\Diff(M)$ equivariant map
	\begin{equation}\label{E:Emb.orbit}
		\Emb(S,M)\to\Gr^{\deco}_{S,\mu}(M), \quad\varphi\mapsto\bigl(\varphi(S),\mathfrak D(\varphi)\cdot\mu\bigr).
	\end{equation}
\end{lemma}

\begin{proof}
	Local smooth sections for the $\Diff(S)$ action on $\Diff(S)\cdot\mu$ give rise to local trivializations of \eqref{E:Diff.S.orbit}, hence, this is a principal $\Diff(S,\mu)$ bundle, cf.~the proof of Lemma~\ref{L:geom1}.
	Combining this with the diffeomorphism
	\[
		\Emb(S,M)=\Emb(S,M)\times_{\Diff(S)}\Diff(S)
	\]
	and the diffeomorphism in \eqref{E:fb.deco.ass}, we conclude that \eqref{E:Emb.orbit} is a principal $\Diff(S,\mu)$ bundle too.
\end{proof}

Let $\Diff_c(M)\cdot(N,\nu)$ denote the $\Diff_c(M)$ orbit through $(N,\nu)\in\Gr^{\deco}_S(M)$.

\begin{lemma}\label{L:deco3}
	Suppose $\Diff(S)$ has good orbit at $\mu\in\mathfrak D(S)$ and suppose $(N,\nu)\in\Gr^{\deco}_{S,\mu}(M)$.
	Then $\Diff_c(M)\cdot(N,\nu)$ is open and closed in $\Gr^{\deco}_{S,\mu}(M)$, the (smooth) $\Diff_c(M)$ action on $\Diff_c(M)\cdot(N,\nu)$ admits local smooth sections, and the canonical forgetful map 
	\[
		\Diff_c(M)\cdot(N,\nu)\to\Diff_c(M)\cdot N,\quad(\tilde N,\tilde\nu)\mapsto\tilde N
	\] 
	is a smooth fiber bundle with typical fiber $\Diff(S)\cdot\mu$.
	
	If, moreover, $\Diff(S)$ has good isotropy at $\mu$, then the group $\Diff_c(M,N,\nu):=\{g\in\Diff_c(M,N):\mathfrak D(g|_N)\cdot\nu=\nu\}$ is a splitting Lie subgroup in $\Diff_c(M,N)$, and the orbit map 
	\begin{equation}\label{E:rrr}
		\Diff_c(M)\to\Diff_c(M)\cdot(N,\nu),\quad g\mapsto\bigl(g(N),\mathfrak D(g|_N)\cdot\nu\bigr)
	\end{equation}
	is a smooth principal $\Diff_c(M,N,\nu)$ bundle.
\end{lemma}

\begin{proof}
	The first part follows immediately from Lemma~\ref{L:deco1}.

	To show the second part, recall that $\Diff_c(M,N)$ is a splitting Lie subgroup in $\Diff_c(M)$ \cite[Lemma~2.1(b)]{HV20} and that the homomorphism
	\[
		\Diff_c(M,N)\to\Diff(N)
	\]
	is a smooth fiber bundle \cite[Lemma~2.1(d)]{HV20}.
	As $\Diff_c(M,N,\nu)$ is the preimage of $\Diff(N,\nu)$ under this homomorphism, and since $\Diff(N,\nu)$ is a splitting Lie subgroup of $\Diff(N)$ by assumption, we conclude that $\Diff_c(M,N,\nu)$ is a splitting Lie subgroup in $\Diff_c(M,N)$.
	Using local smooth sections for the $\Diff_c(M)$ action on $\Diff_c(M)\cdot(N,\nu)$, we obtain local trivializations of \eqref{E:rrr}.
	Hence, this is a principal bundle.
\end{proof}

Let $G$ be a Lie group acting smoothly on $M$ as in Section \ref{SS:geom}.
For $(N,\nu)\in\Gr^{\deco}_{S,\mu}(M)$, we let $G\cdot(N,\nu)$ denote the $G$ orbit through $(N,\nu)$ and we let $G_{(N,\nu)}=\{g\in G_N:g\cdot\nu=\nu\}$ denote the isotropy group.

The following reduces to Lemma~\ref{L:deco3} if $G=\Diff_c(M)$.

\begin{lemma}\label{L:deco5}
	Suppose $\mu\in\mathfrak D(S)$ and $(N,\nu)\in\Gr^{\deco}_{S,\mu}(M)$.
	Assume $\Diff(S)$ has good orbit at $\mu$ and $G$ has good orbit and good extension at $N$.
	Then $G$ has good orbit at $(N,\nu)$, i.e., $G\cdot(N,\nu)$ is an initial splitting smooth submanifold in $\Gr^{\deco}_{S,\mu}(M)$ and the (smooth) $G$ action on $G\cdot(N,\nu)$ admits local smooth sections.
	Moreover, the $G$ equivariant forgetful map
	\begin{equation}\label{E:bcd}
		G\cdot(N,\nu)\to G\cdot N,\quad(\tilde N,\tilde\nu)\mapsto\tilde N,
	\end{equation}
	is a smooth fiber bundle with typical fiber $\Diff(S)\cdot\mu$.
	
	If, moreover, $\Diff(S)$ has good isotropy at $\mu$ and $G_N$ has good isotropy at $\iota_N$, 
	then $G_N$ has good isotropy at $\nu$, i.e., $G_{(N,\nu)}$ is a splitting Lie subgroup in $G_N$.
	Furthermore, the orbit map
	\begin{equation}\label{E:orbit.Nnu}
		G\to G\cdot(N,\nu),\quad g\mapsto\bigl(g\cdot N,\mathfrak D(g|_N)\cdot\nu\bigr),
	\end{equation}
	is a smooth principal bundle with structure group $G_{(N,\nu)}$.
\end{lemma}

\begin{proof}
	For the first part we can proceed as in the proof of Lemma~\ref{L:deco3}.
	Alternatively, we can combine Lemma~\ref{L:aug2} with Lemma~\ref{L:deco.aug}(b,c) below to obtain the first assertion.
	Using Lemma~\ref{L:geom2} one readily generalizes the proof of the second part of Lemma~\ref{L:deco3}.
	Alternatively, we can combining Lemma~\ref{L:aug3} with Lemma~\ref{L:deco.aug}(b,c,d) to obtain the second assertion.
\end{proof}

\begin{example}[Volume densities]\label{Ex:deco.den}
	The functor $\mathfrak D(S):=\Den(S)=\Gamma(|\Lambda|_S)$  is a decoration functor.
	The functor $\mathfrak D(S):=\Den_\times(S)=\Gamma(|\Lambda|_S\setminus S)$ is a decoration functor 
	with good orbit and good isotropy at each point  (see Proposition \ref{P:den}).
	The latter gives rise to the \emph{weighted nonlinear Grassmannian} $\Gr^{\wt}_S(M)$.
	The decoration functor $\mathfrak D(S):=\Gamma(\Lambda^{\dim S}(T^*S)\setminus S)$ of volume forms has similar properties.
\end{example}

\begin{example}\label{Ex:deco.bad}
	The functors $\mathfrak D(S):=\mathfrak X(S)$, $\mathfrak D(S):=\Omega^q(S)=\Gamma(\Lambda^qT^*S)$, and $\mathfrak D(S):=\Omega^q(S;\mathfrak o_S)=\Gamma(\Lambda^q(T^*S)\otimes\mathfrak o_S)$ are decorations functors, but they rarely have good orbits or isotropy.
	Still, the decoration functor $\mathfrak D(S):=\Omega^1(S)$ has good orbits at each nowhere zero closed 1-form (see Proposition \ref{P:begood}).	
\end{example}

\begin{example}[Orientations]\label{Ex:deco.or}
	The functor $\mathfrak D(S):=\Gamma(\tilde S)$, where $\tilde S$ denotes the orientation covering of $S$, 
	is a decoration functor with good orbit and good isotropy at each point.
	It gives rise to the \emph{oriented nonlinear Grassmannian} $\Gr^{\ori}_S(M)$, consisting of all oriented submanifolds of $M$.
	The forgetful functor $\Gr^{\ori}_S(M)\to \Gr_S(M)$ is a finite covering over the (open and closed) subset of orientable submanifolds in $\Gr_S(M)$.
\end{example}

\begin{example}[Cohomology classes]\label{Ex:deco.homol}
	The functor $\mathfrak D(S):=H^q(S;\mathbb R)$ is a decoration functor with good orbit and good isotropy at each point.
	The same is true for $\mathfrak D(S):=H^q(S;\mathfrak o_S)$.
	For $q=\dim S$, we obtain the \emph{homologically weighted nonlinear Grassmannian} $\Gr^{\hwt}_S(M)$, cf.~\cite{HV23}.
\end{example}

\begin{example}\label{Ex:deco.label}
	Let $n_S$ denote the number of connected components of $S$. The functor $\mathfrak D(S):=\Bij(\pi_0(S),\{1,\dotsc,n_S\})$ is a decoration functor with good orbit and good isotropy at each point.
	The corresponding decorated Grassmannian consists of submanifolds of $M$ of type $S$ with labeled connected components.
\end{example}

\begin{example}[Submanifolds]\label{Ex:deco.flag}
	If $S'$ is a closed manifold, then the functor $\mathfrak D(S):=\Gr_{S'}(S)$ is a decoration functor with good orbit and good isotropy at each point, cf.~\cite[Lemma~2.1(b)]{HV20}.
	The corresponding decorated Grassmannian coincides with the Fr\'echet manifold of \emph{nonlinear flags of type $(S',S)$} in $M$, cf.~\cite[Section~2.3]{HV20}.
\end{example}

\begin{example}[Gauge groupoid]\label{Ex:deco.emb}
	The functor $\mathfrak D(S):=\Emb(S,M)$ is a decoration functor with good orbit and trivial isotropy at each embedding $\ph$.
	The corresponding decorated Grassmannian coincides with the Fr\'echet manifold 
	\[
	\Emb(S,M)\x_{\Diff(S)}\Emb(S,M)
	=\{g\in\Diff(N_1,N_2):N_1,N_2\in\Gr_S(M)\},
	\]
	a Lie groupoid,	namely the gauge groupoid of the principal bundle in~\eqref{emb}.
\end{example}


\subsection{Isotropic submanifolds decorated with volume densities}
Throughout this section $(M,\om)$ denotes a symplectic manifold.
We recall results from \cite{W90,L09,GBV19} about coadjoint orbits of the Hamiltonian diffeomorphism group $\Ham_c(M)$, 
parametrized by weighted isotropic nonlinear Grassmannians in $M$, extended to a possibly nonconnected model manifold $S$ in \cite{HV23}. 

On a closed manifold, we consider the decoration functor of volume densities from Example \ref{Ex:deco.den}:
\[
	\DD(S):=\Den_\x(S)=\Ga(|\La|_S\setminus S).
\]
The associated decorated Grassmannian is the weighted nonlinear Grassmannian
$\Gr^{\wt}_{S}(M):=\{(N,\nu):N\in\Gr_S(M),\nu\in\Den_\x(N)\}$.
There is a natural injective $\Diff(M,\omega)$ equivariant map 
\begin{equation}\label{bird}
	J:\Gr_S^{\wt}(M)\to \ham_c(M)^*=C_0^\oo(M)^*,\quad \langle J(N,\nu),X_h\rangle=\int_Nh\nu.
\end{equation}

\begin{proposition}\label{P:den}
	$\Diff(S)$ has good orbit and good isotropy at any $\mu\in\Den_\x(S)$.
\end{proposition}

\begin{proof}
	Let us fix a volume density $\mu$ on the $k$-dimensional closed manifold $S$.
	We consider the $\Diff(S)$ equivariant map taking values in the de Rham cohomology with coefficients in the orientation bundle $\mathfrak o_S$:
	\[
		h_S:\Den_\x(S)\to H^k(S;\mathfrak o_S),\quad h_S(\al)=[\al].
	\]
	It is well known \cite{M65} that the $\Diff(S)_0$ orbit of $\mu$ is the convex subset of all
	volume densities on $S$ that represent the same cohomology class as $[\mu]$.
	Hence, the orbit $\Diff(S)\cdot\mu$ coincides with the set of all volume densities on $S$ that are in the preimage  under $h_S$
	of  the (finite) orbit $\Diff(S)\cdot[\mu]$ in $H^k(S;\mathfrak o_S)$.
	This is an open subset in a finite union of parallel closed affine subspaces
	with finite codimension $\dim H^k(S;\mathfrak o_S)$. 
	In particular, $\Diff(S)\cdot\mu$ is a splitting smooth submanifold in $\Den_\x(S)$ with finite codimension 
	$\dim H^k(S;\mathfrak o_S)$ and with tangent spaces $\ker h_S = d\Om^{k-1}(S;\mathfrak o_S)$.

	That the $\Diff(S)$ action on the orbit $\Diff(S)\cdot\mu$ admits local smooth sections
	is a special case of \cite[Lemma~2.12]{HV23}.
	Here we give a direct proof by using Lemma~\ref{L:loc.sec}.
	Let $\phi$ be a smooth right inverse to the differential $d:\Om^{k-1}(S;\mathfrak o_S)\to d\Om^{k-1}(S;\mathfrak o_S)$.
	For the tangent vector at $\tilde\mu\in \Diff(S)\cdot\mu$ with $\la\in d\Om^{k-1}(S;\mathfrak o_S)$,
	we define $\si(\tilde\mu,\la)=Z\in\X(S)$  by $i_Z\tilde\mu=-\phi(\la)$.
	Since $\la=d\phi(\la)=-L_Z\tilde\mu$ is the infinitesimal generator of $Z$ at $\tilde\mu$, 
	the smooth map $\si:T(\Diff(S)\cdot\mu)\to\X(S)$ permits to apply Lemma~\ref{L:loc.sec}.

	This shows that $\Diff(S)$ has good orbit at $\mu$.
	On the other hand it is well known that the group $\Diff(S,\mu)$ of diffeomorphisms preserving $\mu$
	is a splitting Lie subgroup in $\Diff(S)$ (see Theorem III.2.5.3 on page 203 in \cite{Ham}).
	Thus, $\Diff(S)$ has good isotropy at $\mu$ too.
\end{proof}

We define the space of weighted submanifolds of type $(S,\mu)$ as in \eqref{esmu}:
\[
	\Gr^{\wt}_{S,\mu}(M):=\{(N,\nu)\in\Gr^{\wt}_S(M):(N,\nu)\cong(S,\mu)\}.
\]
Combining Proposition~\ref{P:den} with Lemma~\ref{L:deco1}, we see that the latter is an initial splitting smooth submanifold in $\Gr^{\wt}_S(M)$.
Moreover, the (smooth) $\Diff_c(M)$ action on $\Gr^{\wt}_{S,\mu}(M)$ admits local smooth sections, and the forgetful map 
\begin{equation}\label{E:pGrSmu}
	\Gr_{S,\mu}^{\wt}(M)\to\Gr_S(M)
\end{equation}
is a smooth fiber bundle with typical fiber $\Diff(S)\cdot\mu$.
Furthermore, the bijection in~\eqref{E:abcd} restricts to a canonical $\Diff(M)$ equivariant diffeomorphism
\begin{equation}\label{E:fb.wt.ass}
	\Gr^{\wt}_{S,\mu}(M)=\Emb(S,M)\times_{\Diff(S)}\Diff(S)\cdot\mu
\end{equation}
and the $\Diff(M)$ equivariant map $\Emb(S,M)\to\Gr^{\wt}_{S,\mu}(M)$, $\varphi\mapsto(\varphi(S),\varphi_*\mu)$
is a principal $\Diff(S,\mu)$ bundle according to Lemma~\ref{L:deco2}.

Since $\Gr^{\iso}_S(M,\omega)$ is a splitting smooth submanifold in $\Gr_S(M)$, its preimage under the fiber bundle map in~\eqref{E:pGrSmu}, $\Gr^{\wt\iso}_{S,\mu}(M,\omega)$,
is a splitting smooth submanifold in $\Gr^{\wt}_{S,\mu}(M)$.
Let us fix a leaf $\L\subseteq\Gr_S^{\iso}(M,\om)$ of Weinstein's isodrastic distribution $\D$ in \eqref{DDDDn}. 
Recall that $\mathcal L$ is an initial splitting smooth submanifold of codimension $\dim H^1(S;\RR)$ in $\Gr^{\iso}_S(M)$, cf.~Proposition~\ref{pp}.
Hence, its preimage under the fiber bundle map in~\eqref{E:pGrSmu},
\begin{equation}\label{smwi}
	\G:=\Gr_{S,\mu}^{\wt\iso}(M,\om)|_{\mathcal L},
\end{equation}
is an initial splitting submanifold of codimension $\dim H^1(S;\RR)$ in $\Gr_{S,\mu}^{\wt\iso}(M,\om)$.
	
Combining Propositions~\ref{pp} and \ref{P:den} with Lemma \ref{L:deco5} we obtain:

\begin{proposition}\label{abo}
	The $\Ham_c(M)$ action on $\G$ admits local smooth sections.
\end{proposition}

Now, for any isotropic submanifold $N\subseteq M$ endowed with a volume density $\nu$, the coadjoint orbit of $J(N,\nu)$, with $J$ in \eqref{bird},
is a good coadjoint orbit of $\Ham_c(M)$.
More precisely, by applying Lemma~\ref{lemac} to connected components of the decorated Grassmannian $\G$ acted on by the regular Lie group $\Ham_c(M)$, we get:

\begin{theorem}[{\cite{L09,W90}}]\label{T:th1}
	Let $(M,\omega)$ be a symplectic manifold and consider a connected component $\G_0$ of the decorated Grassmannian 
	$\Gr_{S,\mu}^{\wt\iso}(M,\om)|_{\mathcal L}$ where $\mathcal L$ denotes an isodrastic leaf in $\Gr_S^{\iso}(M,\om)$.
	Then its image $J(\G_0)$ under the map $J$ in~\eqref{bird} is a good coadjoint orbit of $\Ham_c(M)$
	and the (formal) pullback of the KKS symplectic form, denoted by $\Om=J^*\om_{\KKS}$,
	is a $\Ham_c(M)$ equivariant smooth (weakly) symplectic form on $\G_0$ characterized by
	\begin{equation}\label{nomega}
		\Om_{(N,\nu)}(\ze_{X},\ze_{Y})=-\int_N\om(X,Y)\nu,
	\end{equation}
	for all $(N,\nu)\in \G_0$ and $X,Y\in\ham_c(M)$.
	Furthermore, $J$ is an (equivariant) moment map for the Hamiltonian action of $\Ham_c(M)$ on $\G_0$.
\end{theorem}

In the special case $H^1(S)=0$, the decorated nonlinear Grassmannian $(\G,\om)$ 
can also be obtained via symplectic reduction at zero on the right leg of the Marsden--Weinstein ideal fluid dual pair \cite{MW83,GBV12},
as shown in \cite{GBV19}.

The diffeomorphism in \eqref{E:fb.wt.ass} restricts to a diffeomorphism
\begin{equation*}
	\Gr_{S,\mu}^{\wt\iso}(M,\om)=\Emb^{\iso}(S,M)\x_{\Diff(S)}\Diff(S)\cdot\mu,
\end{equation*}
with $\Emb^{\iso}(S,M)$ denoting the manifold of isotropic embeddings.
The pullback  of Weinstein's isodrastic distribution $\D$ in \eqref{DDDDn} to  $\Emb^{\iso}(S,M)$,
\[
	\bar{\mathcal{D}}_\ph:=\bigl\{u_\ph\in\Ga(\ph^*TM):\ph^*i_{u_\ph}\om\in dC^\oo(S)\bigr\},
\]
has codimension $\dim H^1(S;\RR)$ and is integrable. 
The $\Ham_c(M)$ action on each leaf of $\bar{\D}$ admits local smooth sections by Proposition \ref{abo} 
combined with the good extension property in Proposition \ref{pp}.
The preimage $\pi^{-1}(\L)\subseteq\Emb^{\iso}(S,M)$ 
under the bundle projection in~\eqref{emb} is a disjoint union of leaves of $\bar\D$, stable under the $\Diff(S)$ action.
We notice that the decorated Grassmannian in \eqref{smwi} is $\G=\pi^{-1}(\L)\x_{\Diff(S)}\Diff(S)\cdot\mu$ and we denote by
\begin{equation}\label{pmap}
	p:\pi^{-1}(\L)\x\Diff(S)\cdot\mu\to \G ,\quad p(\ph,\mu)=(\ph(S),\ph_*\mu)
\end{equation}
the principal $\Diff(S)$ bundle projection.
This permits to formulate the following alternative characterization of the symplectic form in \eqref{nomega}:

\begin{corollary}[{\cite[Theorem~5]{L09}}]\label{noua} 
	The pullback $\bar\Om:=p^*\Om$ of the symplectic form on $\G$  to the submanifold 
	$\pi^{-1}(\L)\x\Diff(S)\cdot\mu\subseteq\Emb^{\iso}(S,M)\x\Den_\x(S)$ is given by 
	\begin{equation}\label{omphal}
		\bar\Om_{(\ph,\mu)}((u_\ph,d\ga),(u'_\ph,d\ga'))=\int_S(\om(u_\ph,u'_\ph)\al-\ph^*i_{u_\ph}\om\wedge\ga'
		+\ph^*i_{u'_\ph}\om\wedge\ga),
	\end{equation}
	where $u_\ph,u'_\ph\in\D_\ph$ and $d\ga,d\ga'\in d\Om^{k-1}(S,\mathfrak o_S)$.
\end{corollary}

\begin{proof}
	By the infinitesimal transitivity, each element of $\D_\ph$ is of the form $u_\ph=X_h\o\ph$ for some $X_h\in\ham_c(M)$,
	and each element of $\ker T_\mu h_S$ is of the form $d\ga=L_Z\mu$ for some $Z\in\X(S)$. 
	Moreover, let $h_Z$ be a Hamiltonian function that satisfies $T\ph\o Z=X_{h_Z}\o\ph$ (which always exists).
	Knowing that $T_{(\ph,\mu)}p(T\ph\o Z,0)=T_{(\ph,\mu)}p(0,L_Z\mu)$, we compute:
	\begin{align*}
		\bar\Om_{(\ph,\mu)}&((u_\ph,d\ga),(u'_\ph,d\ga'))=\bar\Om_{(\ph,\mu)}((X_h\o\ph,L_Z\mu),(X_{h'}\o\ph,L_{Z'}\mu))\\
		&=\Om_{p(\ph,\mu)}(Tp(X_h\o\ph+T\ph\o Z,0),Tp(X_{h'}\o\ph+T\ph\o Z',0))\\
		&=\Om_{(\ph(S),\ph_*\mu)}(\ze_{X_h+X_{h_Z}},\ze_{X_{h'}+X_{h_{Z'}}})
		\stackrel{\eqref{nomega}}{=}-\int_S\ph^*(\om(X_{h}+X_{h_Z},X_{h'}+X_{h_{Z'}}))\mu.
	\end{align*}
	One of the four terms,  $\ph^*(\om(X_{h_Z},X_{h_{Z'}}))=(\ph^*\om)(Z,Z')$, vanishes because $\ph$ is an isotropic embedding.
	Another term is $\ph^*(\om(X_{h_Z},X_{h'}))=-i_{Z}\ph^*i_{X_{h'}}\om= -i_{Z}\ph^*i_{u'_\ph}\om$.
	Since $d(i_Z\mu-\ga)=0$ and $\ph^*i_{u_\ph}\om$ is exact, the remaining three terms become:
	\begin{multline*}
		\int_S\ph^*\om(X_h,X_{h'})\mu
		-\int_S i_{Z'}\ph^*i_{u_\ph}\om\mu+\int_S i_Z\ph^*i_{u'_\ph}\om\mu\\
		=\int_S\om(u_\ph,u'_\ph)\mu-\int_S \ph^*i_{u_\ph}\om\wedge\ga'+\int_S \ph^*i_{u'_\ph}\om\wedge\ga,
	\end{multline*}
	which is formula~\eqref{omphal}.
\end{proof}


\subsection{Hypersurfaces decorated with nowhere zero closed 1-forms}\label{sheets}
Let $M$ be a manifold endowed with a volume form $\mu$, and let $\Diff_{c,\ex}(M)$ be the group of exact volume preserving diffeomorphisms.
Here, we study decorated nonlinear Grassmannians of codimension one submanifolds endowed with closed 1-forms, 
which parametrize coadjoint orbits of vortex sheets, also considered in \cite{goldin,goldin2,khesin,GBV24}.
Coadjoint orbits of vortex filaments, which are parametrized by codimension two (plain) nonlinear Grassmannians, have been described in Theorem \ref{P:codim2}.

In Example~\ref{Ex:deco.bad} we mention decoration functors of differential forms.
Here we need the decoration functor $\DD(S)=Z_\x^1(S)$ of nowhere zero closed 1-forms, which yields the decorated nonlinear Grassmannian
\[
\Gr^{\deco}_{S}(M):=\{(N,\nu):N\in\Gr_S(M),\nu\in Z^1_\x(N)\}.
\]
Let $\Gr^{\deco,\hnull}_S(M)$ denote the (open and closed) subset in $\Gr^{\deco}_S(M)$ consisting of (decorated) null-homologous submanifolds in $M$.
Then there is an injective $\Diff(M,\mu)$ equivariant map
\begin{equation}\label{momentum} 
	J\colon\Gr^{\deco,\hnull}_S(M) \rightarrow \X_{c.\ex}(M)^\ast,
	\quad\left\langle J(N,\nu), X \right\rangle =\int_N  \io_N^*\alpha\wedge\nu,
\end{equation}
where $\al$ is any potential form for $X\in \X_{c.\ex}(M)^\ast$, i.e., $i_X\mu=d\al$.
Since $N$ is null-homologous, the integral in \eqref{momentum} does not depend on the choice of $\al$.

\begin{proposition}\label{P:begood}
	$\Diff(S)$ has good orbit at any $\be\in Z_\x^1(S)$.
	If, moreover, the period group of $\be$ is discrete, then $\Diff(S)$ has good isotropy at $\be$.
\end{proposition}

\begin{proof}
Let $h_S:Z_\x^1(S)\to H^1(S;\R)$ be the $\Diff(S)$ equivariant map given by $h_S(\be)=[\be]$.
	With the Moser trick \cite{M65}, we show that the orbit $\Diff(S)_0\cdot\be\subseteq Z^1_\x(S)$ 
	is the connected component of $\be$ in $h_S^{-1}([\be])$.
	Indeed, let $\be_t\in Z^1_\x(S)$, $t\in [0,1]$, be a smooth curve starting at $\be$ and such that $[\be_t]=[\be]$.
	Then $\dot\be_t$ is exact, so there exists  $f_t\in C^\oo(S)$, smoothly depending on $t$, such that $df_t=\dot\be_t$.
	A smooth vector field $Z_t$ on $S$ satisfying $i_{Z_t}\be_t=-f_t$ can be written with the help of a Riemannian metric on $S$ as
	$Z_t=-\tfrac{f_t}{||\be_t||^2}\be_t^\sharp$.
	In particular, $L_{Z_t}\be_t+\dot\be_t=0$, so the smooth curve $g_t$ in $\Diff(S)$ that integrates $Z_t$ satisfies $g_t^*\be_t=\be$.

	We have shown that the orbit $\Diff(S)_0\cdot\be$ is an open subset of a closed affine subspace of $Z^1(S)$ with finite codimension $\dim H^1(S;\RR)$,
	thus an initial splitting smooth submanifold with tangent space $\ker h_S=dC^\oo(S)$.
	We use Lemma \ref{L:loc.sec} to show that the $\Diff(S)_0$ action on this orbit  admits local smooth sections.
	With a smooth right inverse $\phi$ to the differential $d:C^\oo(S)\to dC^\oo(S)$ and using a Riemannian metric as above,
	for the  tangent vector at $\tilde\be\in\Diff(S)_0\cdot\be$ with $\la\in dC^\oo(S)$, 
	we define $\si(\tilde\be,\la)=Z\in\X(S)$ such that $i_Z\tilde\be=-\phi(\la)$.
	Since $\la=d\phi(\la)=-L_Z\tilde\be$ is the infinitesimal generator of $Z$ at $\tilde\be$,
	this provides the smooth map $\si:T(\Diff(S)_0\cdot\be)\to\X(S)$ that permits to apply Lemma \ref{L:loc.sec}.
	Hence, the $\Diff(S)_0$ action on $\Diff(S)_0\cdot\be$ admits local smooth sections.

	More generally, the orbit $\Diff(S)\cdot\be$ is  the preimage  under $h_S$ of  the (finite) orbit $\Diff(S)\cdot[\be]$ in $H^1(S;\RR)$, 
	thus an open subset in a finite union of parallel closed affine subspaces with finite codimension and tangent spaces $dC^\oo(S)$. 
	It follows that also $\Diff(S)\cdot\be$ is a splitting smooth submanifold of $\D(S)$ with the  $\Diff(S)$ action admitting local smooth sections,
	which means that $\Diff(S)$ has good orbit at $\be$.

	To show that $\Diff(S)$ has good isotropy at $\be$ with discrete period group, we need to check that $\Diff(S,\be)$ is a Lie group.
	The condition on $\be$ ensures the existence of a fibration $p:S\to B$ with base manifold $B\cong S^1$ and $\be=p^*d\th$.
	The diffeomorphism group $\Diff(S,\be)$ is the preimage of the rotation group of $B$ under the fiber bundle projection $\Aut(S)\to\Diff(B)$,
	where $\Aut(S)$ denotes the group of bundle automorphisms.
\end{proof}

By applying the Proposition~\ref{P:begood} together with Lemma \ref{L:deco1}, we get that
\[
	\Gr^{\deco}_{S,\be}(M)
	:=\{(N,\nu)\in\Gr^{\deco}_S(M):(N,\nu)\cong(S,\be)\}
\]
is an initial splitting smooth submanifold in $\Gr^{\deco}_S(M)$.
Moreover, the (smooth) $\Diff_c(M)$ action on $\Gr^{\deco}_{S,\be}(M)$ admits local smooth sections 
and the forgetful map
\begin{equation}\label{E:pGrbeta}
	\Gr^{\deco}_{S,\be}(M)\to\Gr_S(M)
\end{equation}
is a smooth fiber bundle with typical fiber $\Diff(S)\cdot\beta$.
Furthermore, the bijection in~\eqref{E:abcd} restricts to a canonical $\Diff(M)$ equivariant diffeomorphism
\begin{equation}\label{E:gds}
	\Gr^{\deco}_{S,\be}(M)=\Emb(S,M)\times_{\Diff(S)}\Diff(S)\cdot\be.
\end{equation}

Let $S$ be a closed manifold with $\dim S=\dim M-1$, and let $\L$ be a leaf of  the isodrastic distribution $\D$ on $\Gr_S(M)$ studied in Section~\ref{S:isovol}.
Since $\mathcal L$ is an initial splitting smooth submanifold of codimension $\dim H^0(S;\ou_S)$ in $\Gr_S(M)$, its preimage under the fiber bundle map in~\eqref{E:pGrbeta},
\begin{equation}\label{swim}
	\G:=\Gr^{\deco}_{S,\be}(M)|_{\mathcal L},
\end{equation}
is an initial splitting smooth submanifold in $\Gr_{S,\be}^{\deco}(M)$ of the same codimension, i.e., the codimension equals the number of orientable connected components of $S$.

By combining Propositions~\ref{P:Diffmu.good.ext}, \ref{L:Fcex.int} and \ref{P:begood} with Lemma \ref{L:deco5} we get:

\begin{proposition}\label{P:qwerty}
	The $\Diff_{c,\ex}(M)$ action on $\G$ admits local smooth sections.
\end{proposition}

Now, for any null-homologous codimension one submanifold $N\subseteq M$ endowed with a nowhere zero closed 1-form $\nu$, 
the coadjoint orbit of $J(N,\nu)$, with $J$ in \eqref{momentum},  is a good coadjoint orbit of  $\Diff_{c,\ex}(M)$.
Assume that the isodrast $\L$ consists of null-homologous submanifolds of $M$.
By applying Lemma~\ref{lemac} to connected components of $\G$ in~\eqref{swim}, acted upon by the regular Lie group $\Diff_{c,\ex}(M)$, 
and noticing that $i_{[X,Y]}\mu=di_Xi_Y\mu$ for all $X,Y\in\X_c(M,\mu)$, we obtain the following:

\begin{theorem}[{\cite{GBV24}}]\label{T:th2}
	Let $\mu$ be a volume form on $M$ and consider a connected component $\G_0$ of the decorated Grassmannian $\Gr^{\deco}_{S,\be}(M)|_{\mathcal L}$ 
	where $\L$ is an isodrast in $\Gr_S(M)$ through a null-homologous submanifold of codimension one.
	Then its image $J(\G_0)$ under the map $J$ in~\eqref{momentum} is a good coadjoint orbit of $\Diff_{c,\ex}(M)$ and
	the (formal) pullback of the KKS symplectic form, denoted by $\om=J^*\om_{\KKS}$,  
	is a $\Diff_{\ex}(M)$ equivariant smooth (weakly) symplectic form on $\G_0$ characterized by
	\begin{equation}\label{nuomega}
		\om_{(N,\nu)}(\ze_{X},\ze_{Y})=\int_N\io_N^*(i_{X}i_{Y}\mu)\wedge\nu,
	\end{equation}
	for all $(N,\nu)\in \G_0$ and $X,Y\in\X_{c,\ex}(M)$.
	Furthermore, $J$ is an (equivariant) moment map for the Hamiltonian action of $\Diff_{c,\ex}(M)$ on $\G_0$.
\end{theorem}

\begin{remark}
	The map \eqref{momentum} is well defined also for more general vortex sheets $(N,\nu)\in\Gr_S^{\deco}(M)$. 
	These are characterized by the property that $\io_N^*H_c^{n-2}(M)\subseteq H^{n-2}(N)$ 
	is contained in the annihilator of $[\nu]\in H^1(N;\RR)$.
	We still get good coadjoint orbits of $\Diff_{c,\ex}(M)$ parametrized by decorated nonlinear Grassmannians (see \cite[Theorem~3.6]{GBV24}).
	Otherwise, one  needs to pass to central extensions \cite[Theorem~3.5]{GBV24}.
\end{remark}

We consider the pullback  to $\Emb(S,M)$ of the  distribution $\D$:
\[
	\bar\D_\ph=\bigl\{u_\ph\in\Ga(\ph^*TM):\ph^*i_{u_\ph}\mu\in \Om^{\dim S}(S)\text{ is exact}\bigr\}.
\]
It follows from Proposition \ref{P:qwerty} and the good extension in Proposition~\ref{P:Diffmu.good.ext} that the $\Diff_{c,\ex}(M)$ action 
on each leaf of $\bar\D$ admits local smooth sections.
The preimage  $\pi^{-1}(\L)$ under the bundle projection \eqref{emb}
is a disjoint union of leaves of  $\bar\D$, stable under the $\Diff(S)$ action.
We notice that diffeomorphism in \eqref{E:gds} restricts to a diffeomorphism $\G=\pi^{-1}(\L)\x_{\Diff(S)}\Diff(S)\cdot\be$ 
where $\G$ is the decorated Grassmannian in~\eqref{swim}.
We denote the corresponding principal $\Diff(S)$ bundle projection by
\begin{equation}\label{pmapp}
	p:\pi^{-1}(\L)\x\Diff(S)\cdot\be\to \G ,\quad p(\ph,\be)=(\ph(S),\ph_*\be).
\end{equation}
A three term formula, similar to that in Corollary~\ref{noua} holds.

\begin{corollary}\label{zece}
	The pullback $\bar\om:=p^*\om$ of the symplectic form on $\M$ to the submanifold 
	$\pi^{-1}(\L)\x\Diff(S)\cdot\be\subseteq\Emb(S,M)\x\Om^1_\x(S)$ is given by 
	\begin{equation*}
		\bar\om_{(\ph,\be)}((u_\ph,df),(u'_\ph,df'))=\int_S\ph^*i_{u_\ph}i_{u'_\ph}\mu\wedge\be+(-1)^{\dim S}\int_S (f\ph^*i_{u'_\ph}\mu- f'\ph^*i_{u_\ph}\mu),
	\end{equation*}
	where $u_\ph,u'_\ph\in\D_\ph$ and $df,df'\in dC^\oo(S)$.
\end{corollary}

\begin{proof}
	By the infinitesimal transitivity, each element of $\D_\ph$ is of the form $u_\ph=X_\al\o\ph$ for some $X_\al\in\X_{c,\ex}(M)$,
	and each tangent vector at $\be$ to the $\Diff(S)$ orbit is of the form $df=L_Z\be$ for some $Z\in\X(S)$. 
	Moreover, let $\al_Z$ be any potential form that satisfies $T\ph\o Z=X_{\al_Z}\o\ph$ (it always exists because $T\ph\o Z\in\D_\ph$).
	Knowing that $T_{(\ph,\be)}p(T\ph\o Z,0)=T_{(\ph,\be)}p(0,L_Z\be)$, we compute:
	\begin{align*}
		\bar\om_{(\ph,\be)}&((u_\ph,df),(u'_\ph,df'))=\bar\om_{(\ph,\be)}((X_\al\o\ph,L_Z\be),(X_{\al'}\o\ph,L_{Z'}\be))\\
		&=\om_{p(\ph,\be)}(Tp(X_h\o\ph+T\ph\o Z,0),Tp(X_{h'}\o\ph+T\ph\o Z',0))\\
		&=\om_{(\ph(S),\ph_*\be)}(\ze_{X_\al+X_{\al_Z}},\ze_{X_{\al'}+X_{\al_{Z'}}})
		\stackrel{\eqref{nuomega}}{=}\int_S\ph^*i_{X_{\al+\al_Z}}i_{X_{\al'+\al_{Z'}}}\mu\wedge\be.
	\end{align*}
	One of the four terms, $\ph^*i_{X_{\al_Z}}i_{X_{\al_{Z'}}}\mu=i_Zi_{Z'}\ph^*\mu$, vanishes by dimensional reasons.
	Because $\ph^*i_{X_{\al_Z}}i_{X_{\al'}}\mu=i_{Z}\ph^*i_{X_{\al'}}\mu= i_{Z}\ph^*i_{u'_\ph}\mu$,
	the remaining three terms become:
	\begin{multline*}
		\int_S\ph^*i_{X_\al}i_{X_{\al'}}\mu\wedge\be
		-\int_S i_{Z'}\ph^*i_{u_\ph}\mu\wedge\be+\int_S i_Z\ph^*i_{u'_\ph}\mu\wedge\be\\
		=\int_S\ph^*i_{u_\ph}i_{u'_\ph}\mu\wedge\be-(-1)^{\dim S}\int_S f'\ph^*i_{u_\ph}\mu+(-1)^{\dim S}\int_S f\ph^*i_{u'_\ph}\mu.
	\end{multline*}
	That yields the desired formula.
\end{proof}

\section{Augmented nonlinear Grassmannians}\label{S:aug}

We continue to consider a Lie group $G$ acting smoothly on a finite dimensional manifold $M$.


\subsection{Framework for augmentations on nonlinear Grassmannians}\label{SS:aug}
We let $\cataug(M,G)$ denote the category whose objects are closed submanifolds in $M$ and the morphisms between two closed submanifolds $N_1$ and $N_2$ are the elements of $G$ mapping $N_1$ diffeomorphically onto $N_2$.
In particular, we have a canonical functor to the category $\catclosed$ of (finite dimensional) closed manifolds and diffeomorphisms:
\begin{equation}\label{E:cats}
	\cataug(M,G)\to\catclosed.
\end{equation}

An \emph{augmentation functor for the $G$ action on $M$} is a covariant functor $\mathfrak A$ from the category $\cataug(M,G)$ into the category of Fr\'echet manifolds and smooth maps, such that the action of the isotropy group $G_N=\{g\in G:g\cdot N=N\}$ on $\mathfrak A(N)$ is smooth, for each closed submanifold $N$ of $M$.
Since we do not want to presume a Lie group structure on $G_N$, the action is only assumed to be smooth along maps into $G_N$ which are smooth into $G$.
More specifically, for any map $g\colon Z\to G_N$ which is smooth when considered as a map into $G$, the map $Z\times\mathfrak A(N)\to\mathfrak A(N)$, $(z,\gamma)\mapsto\mathfrak A(g(z))\cdot\gamma$ is assumed to be smooth.

Suppose $\mathfrak A$ is an augmentation functor for the $G$-action on $M$.
The set
\[
	\Gr^{\aug}(M):=\bigl\{(N,\gamma):N\in\Gr(M), \gamma\in\mathfrak A(N)\bigr\}
\]
will be referred to as \emph{$\mathfrak A$-augmented nonlinear Grassmannian of $M$}.
The group $G$ acts naturally on $\Gr^{\aug}(M)$.
More explicitly, this action is given by 
\begin{equation}\label{E:action}
	g\cdot(N,\gamma)=\bigl(g\cdot N,\mathfrak A(g)\cdot\gamma\bigr)
\end{equation}
where $g\in G$ and $ \gamma\in\mathfrak A(N)$.
The canonical forgetful map 
\begin{equation}\label{E:Gr.aug.Gr}
	\Gr^{\aug}(M)\to\Gr(M),\qquad(N,\gamma)\mapsto N
\end{equation}
is equivariant over the homomorphism $G\to\Diff(M)$.
For $N\in\Gr(M)$ we let 
\[
	\Gr^{\aug}(M)|_{G\cdot N}=\bigl\{(\tilde N,\tilde\gamma):\tilde N\in G\cdot N,\tilde\gamma\in\mathfrak A(\tilde N)\bigr\}
\]
denote the preimage of the orbit $G\cdot N$ in $\Gr(M)$ under the map in \eqref{E:Gr.aug.Gr}.

\begin{lemma}\label{L:aug1}
	If $G$ has good orbit at $N\in\Gr(M)$, then there exists a Fr\'echet manifold structure on $\Gr^{\aug}(M)|_{G\cdot N}$ such that the $G$ equivariant forgetful map
	\begin{equation}\label{E:aug.fb}
		\Gr^{\aug}(M)|_{G\cdot N}\to G\cdot N,\qquad(\tilde N,\tilde\gamma)\mapsto\tilde N
	\end{equation}
	is a smooth fiber bundle with typical fiber $\mathfrak A(N)$ and such that the $G$ action on $\Gr^{\aug}(M)|_{G\cdot N}$ is smooth.

	If, moreover, $G$ has good isotropy at $N$, then the $G$ equivariant identification
	\begin{equation}\label{E:abc}
		\Gr^{\aug}(M)|_{G\cdot N}=G\times_{G_N}\mathfrak A(N),\qquad \bigl(g\cdot N,\mathfrak A(g)\cdot\gamma\bigr)\leftrightarrow[g,\gamma]
	\end{equation}
	is a diffeomorphism, where the right hand side denotes the bundle associated to the smooth principal bundle in \eqref{E:princ.sigma} and the smooth action of $G_N$ on $\mathfrak A(N)$.
\end{lemma}

\begin{proof}
	Since $G$ has good orbit at $N$, there exists an open cover $U_i$ of $G\cdot N$ and smooth maps $u_i\colon U_i\to G$ such that $u_i(\tilde N)\cdot N=\tilde N$ for all $\tilde N\in U_i$.
	We obtain a trivialization for the map in~\eqref{E:aug.fb} over $U_i$ by defining
	\begin{equation}\label{E:triv.aug}
		\phi_i\colon U_i\times\mathfrak A(N)\to\Gr^{\aug}(M)|_{U_i},\qquad
		\phi_i(\tilde N,\gamma)=\bigl(\tilde N,\mathfrak A(u_i(\tilde N))\cdot\gamma\bigr)
	\end{equation}
	where $\Gr^{\aug}(M)|_{U_i}$ denotes the preimage of $U_i$ under the forgetful map in~\eqref{E:aug.fb}.
	The inverse of $\phi_i$ is given by $\phi_i^{-1}(\tilde N,\tilde\gamma)=(\tilde N,\mathfrak A(u_i(\tilde N)^{-1})\cdot\tilde\gamma)$ where $(\tilde N,\tilde\gamma)\in\Gr^{\aug}(M)|_{U_i}$.
	For the transition $\phi_j^{-1}\circ\phi_i\colon(U_i\cap U_j)\times\mathfrak A(N)\to(U_i\cap U_j)\times\mathfrak A(N)$ we find
	\begin{equation}\label{E:trans.aug}
		(\phi_j^{-1}\circ\phi_i)(\tilde N,\gamma)=\bigl(\tilde N,\mathfrak A\bigl(u_j(\tilde N)^{-1}u_i(\tilde N)\bigr)\cdot\gamma\bigr).
	\end{equation}
	These transitions are smooth, since the map $\tilde N\mapsto u_j(\tilde N)^{-1}u_i(\tilde N)$ takes values in $G_N$ and is smooth into $G$.
	Hence, the trivializations $\phi_i$ provide a smooth fiber bundle atlas for the map in~\eqref{E:aug.fb}.
	In particular, this turns $\Gr^{\aug}(M)|_{G\cdot N}$ into a Fr\'echet manifold.

	Put $V_{ij}:=\{(g,\tilde N)\in G\times U_i:g\cdot\tilde N\in U_j\}$.
	Then $V_{ij}$ is open in $G\times U_i$ and the $G$ action on $G\cdot N$ restricts to a smooth map $V_{ij}\to U_j$.
	Expressing the $G$ action on $\Gr^{\aug}(M)|_{G\cdot N}$ in the local trivializations $\phi_i$ and $\phi_j$, 
	we obtain a map $V_{ij}\times\mathfrak A(N)\to U_j\times\mathfrak A(N)$ given by
	\[
		(g,\tilde N,\gamma)\mapsto\bigl(g\cdot\tilde N,\mathfrak A\bigl(u_j(g\cdot\tilde N)^{-1}gu_i(\tilde N)\bigr)\cdot\gamma\bigr).
	\]
	The latter is smooth since the map $V_{ij}\to G$, $(g,\tilde N)\mapsto u_j(g\cdot\tilde N)^{-1}gu_i(\tilde N)$ takes values in $G_N$ and is smooth into $G$.
	This shows that $G$ acts smoothly on $\Gr^{\aug}(M)|_{G\cdot N}$.

	To prove the second part of the lemma we assume, moreover, that $G$ has good isotropy at $N$.
	Hence, $G_N$ is a splitting Lie subgroup in $G$ and $u_i\colon U_i\to G$ is a local section of the principal $G_N$ bundle $G\to G\cdot N$.
	The corresponding local trivialization of the associated bundle $G\times_{G_N}\mathfrak A(N)$ over $U_i$ is 
	\[
		U_i\times\mathfrak A(N)\to\bigl(G\times_{G_N}\mathfrak A(N)\bigr)|_{U_i},\qquad(\tilde N,\gamma)\mapsto[u_i(\tilde N),\gamma].
	\]
	Up to the identification in \eqref{E:abc} this coincides with the local trivialization $\phi_i$ in \eqref{E:triv.aug}.
	Hence, the map in \eqref{E:abc} is a diffeomorphism.
\end{proof}

For $\gamma\in\mathfrak A(N)$ we let $G_N\cdot\gamma$ denote the $G_N$ orbit through $\gamma$ in $\mathfrak A(N)$.
We say \emph{$G_N$ has good orbit at $\gamma$} if $G_N\cdot\gamma$ is an initial splitting smooth submanifold in $\mathfrak A(N)$ and the (smooth) $G_N$ action on $G_N\cdot\gamma$ admits local smooth sections.
As we do not presume a Lie group structure on $G_N$, the local sections are only assumed to be smooth into $G$. 
The $G$ orbit through $(N,\gamma)$ in $\Gr^{\aug}(M)$ will be denoted by $G\cdot(N,\gamma)$.

\begin{lemma}\label{L:aug2}
	Suppose $G$ has good orbit at $N\in\Gr(M)$ and suppose $G_N$ has good orbit at $\gamma\in\mathfrak A(N)$.
	Then $G$ has good orbit at $(N,\gamma)$, i.e., $G\cdot(N,\gamma)$ is an initial splitting smooth submanifold in $\Gr^{\aug}(M)|_{G\cdot N}$ 
	and the (smooth) $G$ action on $G\cdot(N,\gamma)$ admits local smooth sections.
	Furthermore, the $G$ equivariant canonical forgetful map
	\begin{equation}\label{E:aug.fbg}
		G\cdot(N,\gamma)\to G\cdot N,\qquad\bigl(\tilde N,\tilde\gamma\bigr)\mapsto\tilde N
	\end{equation}
	is a smooth fiber bundle with typical fiber $G_N\cdot\gamma$.

	If, moreover, $G$ has good isotropy at $N$, then the $G$ equivariant identification
	\begin{equation}\label{E:qwerty}
		G\cdot(N,\gamma)=G\times_{G_N}(G_N\cdot\gamma),\qquad\bigl(g\cdot N,\mathfrak A(g)\cdot\hat\gamma\bigr)\leftrightarrow[g,\hat\gamma]
	\end{equation}
        is a diffeomorphism, where the right hand side denotes the bundle associated to the smooth principal bundle in \eqref{E:princ.sigma} and the smooth action of $G_N$ on $G_N\cdot\gamma$.
\end{lemma}

\begin{proof}
	The local trivializations in \eqref{E:triv.aug} restrict to local trivializations
	\begin{equation}\label{E:aug.fbg.trivi}
		\phi_i\colon U_i\times(G_N\cdot\gamma)\to(G\cdot(N,\gamma))|_{U_i},\quad
		(\tilde N,\hat\gamma)\mapsto\bigl(\tilde N,\mathfrak A(u_i(\tilde N))\cdot\hat\gamma\bigr),
	\end{equation}
	for we have $(\tilde N,\mathfrak A(u_i(\tilde N))\cdot\hat\gamma)=u_i(\tilde N)g\cdot(N,\gamma)$ 
	by \eqref{E:action} if $\hat\gamma=g\cdot\gamma$ with $g\in G_N$.
	Hence, $G\cdot(N,\gamma)$ is an initial splitting smooth submanifold of $\Gr^{\aug}(M)|_{G\cdot N}$.
	Since $G_N\cdot\gamma$ is an initial submanifold in $\mathfrak A(N)$, the transitions in \eqref{E:trans.aug} 
	restrict to diffeomorphisms of $(U_i\cap U_j)\times(G_N\cdot\gamma)$.
	Consequently, the local trivializations in~\eqref{E:aug.fbg.trivi} provide a smooth fiber bundle atlas for the map in~\eqref{E:aug.fbg}.
	As $G\cdot(N,\gamma)$ is an initial submanifold in $\Gr^{\aug}(M)|_{G\cdot N}$, the $G$ action on $\Gr^{\aug}(M)|_{G\cdot N}$ restricts to a smooth action on $G\cdot(N,\gamma)$.

	Since the  $G$ action on $G\cdot(N,\gamma)$ is smooth and transitive, it suffices to construct a local smooth section at $(N,\gamma)$.
	To this end let $U\subseteq G\cdot N$ be an open neighborhood of $N$ and suppose $u\colon U\to G$ is a local smooth section for the $G$ action, 
	i.e., $u(\tilde N)\cdot N=\tilde N$ for all $\tilde N\in U$ and $u(N)=\id$.
	Moreover, let $V\subseteq G_N\cdot\gamma$ be an open neighborhood of $\gamma$ and suppose $v\colon V\to G_N$ is a local smooth section for the $G_N$ action,
	i.e., $v$ is smooth into $G$ and $v(\hat\gamma)\cdot\gamma=\hat\gamma$ for all $\hat\gamma\in V$.
	Then $W:=\{(\tilde N,\tilde\gamma):\tilde N\in U,\tilde\gamma\in\mathfrak A(\tilde N),\mathfrak A(u(\tilde N)^{-1})\cdot\tilde\gamma\in V\}$ 
	is an open neighborhood of $(N,\gamma)$ in $G\cdot(N,\gamma)$ and
	\[
		w\colon W\to G,\quad
		w(\tilde N,\tilde\gamma):=u(\tilde N)v\bigl(\mathfrak A(u(\tilde N)^{-1})\cdot\tilde\gamma\bigr)
	\]
	is a local smooth section for the $G$ action on $G\cdot(N,\gamma)$, i.e., $w(\tilde N,\tilde\gamma)\cdot(N,\gamma)=(\tilde N,\tilde\gamma)$ for all $(\tilde N,\tilde\gamma)\in W$.
	Indeed, the local trivialization associated with $u$, cf.~\eqref{E:aug.fbg.trivi},
	\begin{equation}\label{E:xyz}
		\phi\colon U\times(G_N\cdot\gamma)\to(G\cdot(N,\gamma))|_U,\quad
		(\tilde N,\hat\gamma)\mapsto\bigl(\tilde N,\mathfrak A(u(\tilde N))\cdot\hat\gamma\bigr)
	\end{equation}
	maps $U\times V$ diffeomorphically onto $W$, and $w\circ\phi\colon U\times V\to G$, $(w\circ\phi)(\tilde N,\hat\gamma)=u(\tilde N)v(\hat\gamma)$ 
	is evidently smooth as $u$ and $v$ are both smooth into $G$.

	To prove the second part of the lemma we assume, moreover, that $G$ has good isotropy at $N$.
	Hence, $G_N$ is a splitting Lie subgroup in $G$ and $u\colon U\to G$ is a local section of the the principal $G_N$ bundle $G\to G\cdot N$.
	The corresponding local trivialization of the associated bundle $G\times_{G_N}(G_N\cdot\gamma)$ over $U$ is 
	\[
		U\times(G_N\cdot\gamma)\to\bigl(G\times_{G_N}(G_N\cdot\gamma)\bigr)|_{U},\qquad
		(\tilde N,\hat\gamma)\mapsto[u(\tilde N),\hat\gamma].
	\]
	Up to the identification in \eqref{E:qwerty}, this coincides with the local trivialization $\phi$ in \eqref{E:xyz}.
	Hence, \eqref{E:qwerty} is a diffeomorphism.
\end{proof}

We say \emph{$G$ has good isotropy at $(N,\gamma)\in\Gr^{\aug}(M)$} if $G_{(N,\gamma)}=\{g\in G:g\cdot(N,\gamma)=(N,\gamma)\}$ is a splitting Lie subgroup of $G$.

\begin{lemma}\label{L:aug3}
	Suppose $G$ has good orbit at $N\in\Gr(M)$, suppose $G_N$ has good orbit at $\gamma\in\mathfrak A(N)$, and suppose $G$ has good isotropy at $(N,\gamma)$.
	Then the $G$ equivariant map provided by the action,
	\begin{equation}\label{E:qq}
		G\to G\cdot(N,\gamma),\qquad
		g\mapsto g\cdot(N,\gamma)
		=\bigl(g\cdot N,\mathfrak A(g)\cdot\gamma\bigr)
	\end{equation}
	is a smooth principal $G_{(N,\gamma)}$ bundle.
	Moreover, $G$ has good isotropy at $N$, $G_N$ has good isotropy at $\gamma$, i.e., $G_{(N,\gamma)}$ is a splitting Lie subgroup of $G_N$, and the $G_N$ equivariant map provided by the action,
	\begin{equation}\label{E:qqq}
		G_N\to G_N\cdot\gamma,\qquad
		g\mapsto\mathfrak A(g)\cdot\gamma
	\end{equation}
	is a smooth principal $G_{(N,\gamma)}$ bundle.
\end{lemma}

\begin{proof}
	According to Lemma~\ref{L:aug2}, the $G$ action on $G\cdot(N,\gamma)$ is smooth and admits local smooth sections.
	Choose an open neighborhood $U$ of $(N,\gamma)$ in $G\cdot(N,\gamma)$ and $u\colon U\to G$ smooth such that 
	$u(\tilde N,\tilde\gamma)\cdot(N,\gamma)=(\tilde N,\tilde\gamma)$ for all $(\tilde N,\tilde\gamma)\in U$.
	Then
	\[
		U\times G_{(N,\gamma)}\to G|_U,\qquad
		\bigl((\tilde N,\tilde\gamma),g\bigr)\mapsto u(\tilde N,\tilde\gamma)g
	\]
	is a diffeomorphism trivializing the map in~\eqref{E:qq} over $U$.
	The inverse map, 
	\[
		G|_U\to U\times G_{(N,\gamma)},\qquad
		g\mapsto\bigl(g\cdot(N,\gamma),u(g\cdot(N,\gamma))^{-1}g\bigr),
	\] is smooth since $G_{(N,\gamma)}$ is an initial submanifold in $G$.
	Using the transitive smooth $G$ action on $G\cdot(N,\gamma)$, we obtain local trivializations of~\eqref{E:qq} around any point in $G\cdot(N,\gamma)$.
	Hence, the map in~\eqref{E:qq} is a smooth principal bundle.
	Restricting this bundle to the splitting smooth submanifold $G_N\cdot\gamma$ in the base, cf.~Lemma~\ref{L:aug2}, we obtain the second assertion.
\end{proof}

We summarize most of the maps discussed above in the following diagram, indicating typical fibers and structure groups next to the arrows:
\[
	\xymatrix{
		G\ar@{->>}[rr]^-{G_{(N,\gamma)}}\ar@/_/@{->>}[drr]_-{G_N}
		&&G\cdot(N,\gamma)\ar[d]^-{G_N\cdot\gamma}\ar@{^(->}[r]
		&\Gr^{\aug}(M)|_{G\cdot N}\ar[d]^-{\mathfrak A(N)}\ar@{^(-->}[r]
		&\Gr^{\aug}(M)\ar@{-->}[d]
		\\
		&&G\cdot N\ar@{=}[r]&G\cdot N\ar@{^(->}[r]&\Gr(M)
	}
\]


\subsection{The full group of diffeomorphisms}\label{S:4.2}
	We now consider $G=\Diff_c(M)$, the full group of compactly supported diffeomorphisms.
	For $N\in\Gr(M)$ we let
	\[
		\mathfrak A(N):=\Gamma\bigl((|\Lambda|_N\otimes T^*M|_N)\setminus N\bigr)
	\]
	denote the space of all nowhere vanishing smooth sections of the bundle $|\Lambda|_N\otimes T^*M|_N$ over $N$.
	We equip $\mathfrak A(N)$ with the Fr\'echet manifold structure it inherits as an open subset in the Fr\'echet space of all sections with the $C^\infty$ topology.
	For $g\in G$ we let $\mathfrak A(g)\colon\mathfrak A(N)\to\mathfrak A(g\cdot N)$ denote the natural map induced by $g$.
	Since this is the restriction of a bounded linear map between Fr\'echet spaces, $\mathfrak A(g)$ is smooth.
	Clearly, $\mathfrak A$ defines a functor from $\cataug(M,G)$ into the category of Fr\'echet manifolds.	

	The augmented nonlinear Grassmannian associated with this functor $\mathcal A$ will be denoted by $\Gr^{\aug}(M)$.
	We have an injective and $\Diff(M)$ equivariant map
	\begin{equation}\label{E:J.Diffc}
		J\colon\Gr^{\aug}(M)\to\mathfrak g^*,\qquad\langle J(N,\gamma),X\rangle=\int_N\gamma(X),\quad X\in\mathfrak g
	\end{equation}
	where $\mathfrak g=\mathfrak X_c(M)$ denotes the Lie algebra of compactly supported vector fields.

	\begin{proposition}\label{P:aug.full.orbit}
		$\mathfrak A$ is an augmentation functor for the $G$ action on $M$.
	\end{proposition}

	\begin{proof}
		It remains to check the smoothness assumption.
		To this end consider $N\in\Gr(M)$.
		Recall that $G_N=\Diff_c(M,N)$ is Lie subgroup in $\Diff_c(M)$, see for instance \cite[Lemma~2.1(b)]{HV20}.
		For $f\in\Diff_c(M,N)$, we let $(T^*f)|_N\colon T^*M|_N\to T^*M|_N$ denote the vector bundle isomorphism over $f|_N$ obtained by applying the cotangent bundle functor, 
		and we let $|\Lambda|(f|_N)\colon|\Lambda|_N\to|\Lambda|_N$ denote the vector bundle homomorphism over $f|_N$ obtained by applying the density functor.
		The map
		\begin{align*}
			\Diff_c(M,N)\times\Gamma\bigl(|\Lambda|_N\otimes T^*M|_N\bigr)&\to\Gamma\bigl(|\Lambda|_N\otimes T^*M|_N\bigr)
			\\
			(f,\ga)&\mapsto\bigl(|\Lambda|(f|_N)\otimes(T^*f)|_N\bigr)\circ\ga\circ (f|_N)^{-1}
		\end{align*}
		is smooth and it restricts to the $G_N$ action on the open subset $\mathfrak A(N)$.
	\end{proof}

	Combining Propositions~\ref{P:DiffcM.good} and \ref{P:aug.full.orbit} with Lemma~\ref{L:aug1}, 
	we obtain a Fr\'echet manifold structure on the augmented Grassmannian $\Gr^{\aug}(M)|_{G\cdot N}$ such that the $G$ action is smooth.
	Since $G\cdot N$ is open and closed in $\Gr(M)$ this yields a smooth structure on the disjoint union, $\Gr^{\aug}(M)$,
	such that the $G$ equivariant forgetful map $\Gr^{\aug}(M)\to\Gr(M)$ is a smooth fiber bundle with fiber $\mathfrak A(N)$ over $N\in\Gr(M)$.

\subsubsection*{Augmentations with zero pullback}	
	For $N\in\Gr(M)$ we have a fiberwise surjective vector bundle homomorphism $T^*M|_N\to T^*N$ given by restriction of covectors.
	Its kernel, the annihilator of $TN$, is isomorphic to the dual of the normal bundle of $N$ in $M$ and will be denoted by $\Ann TN=(TM|_N/TN)^*$.
	Hence,
	\begin{equation}\label{E:Ann}
		\Gamma\bigl((|\Lambda|_N\otimes\Ann TN)\setminus N\bigr)\subseteq\mathfrak A(N)
	\end{equation}
	is a splitting smooth submanifold.

	\begin{proposition}\label{P:aug.full.iso}
		Let $G=\Diff_c(M)$. Suppose $N\in\Gr(M)$ and $\gamma\in\Gamma\bigl((|\Lambda|_N\otimes\Ann TN)\setminus N\bigr)$.
		Then $G_N$ has good orbit at $\gamma$.
		Moreover, this orbit is open and closed in $\Ga((|\Lambda|_N\otimes\Ann TN)\setminus N)$.
		Furthermore, $G$ has good isotropy at $(N,\gamma)$.
	\end{proposition}

	\begin{proof}
		Let $\iota_N\colon N\to M$ denote the inclusion and let $\Diff_c(M;\iota_N)$ denote the group of diffeomorphisms preserving $N$ pointwise.
		This is a splitting Lie subgroup of $\Diff_c(M)$, see \cite[Lemma~2.1(c)]{HV20}.
		The action of this group on $\Gamma(\Ann TN\setminus N)$ factors through the homomorphism
		\[
			\Diff_c(M;\iota_N)\to\Gamma\bigl(\GL(TM|_N/TN)\bigr),\qquad g\mapsto Tg|_N.
		\]
		Using a tubular neighborhood of $N$ in $M$, one readily shows that this homomorphism admits a smooth local right inverse at the identity.
		Moreover, $\Diff_c(M;\iota_N)$ acts trivially on $|\Lambda|_N$.
		Hence, denoting the normal bundle by $E:=TM|_N/TN$, it suffices to show that the action of $\Gamma(\GL(E))$ on $\Gamma(E^*\setminus N)$ admits local smooth sections.
		To this end, we choose a fiberwise Euclidean inner product on $E$ and we let $\sharp^{-1}=\flat\colon E\to E^*$ denote the corresponding isomorphism of vector bundles.
		Suppose $\alpha\in\Gamma(E^*\setminus N)$.
		For $\beta\in\Gamma(E^*)$, we define a vector bundle endomorphism $A_\beta\colon E\to E$ by $A_\beta(\sharp\alpha)=\sharp\beta$ and $A_\beta|_{\ker\alpha}=\id$.
		Clearly, $A_\beta$ depends smoothly on $\beta$.
		By construction we have $\alpha A^*_\beta=\flat A_\beta\sharp\alpha=\beta$, where $A_\beta^*\colon E\to E$ denotes the fiberwise adjoint.
		Consider the open neighborhood $U=\{\beta\in\Gamma(E^*):\al(\sharp\be)\neq0\}$ of $\alpha$ in $\Gamma(E^*\setminus N)$.
		If $\beta\in U$, then $A_\beta$ is an isomorphism.
		Hence, $u\colon U\to\Gamma(\GL(E))$, $u(\beta):=(A^*_\beta)^{-1}$ is the desired local smooth section around $\alpha$.

		To prove that $G$ has good isotropy at $(N,\gamma)$ it suffices to show that $G_{(N,\gamma)}=\Diff_c(M,N,\gamma)$ is a splitting Lie subgroup in $G_N=\Diff_c(M,N)$, because the latter is a splitting Lie subgroup in $G=\Diff_c(M)$, cf.~Proposition~\ref{P:DiffcM.good}.
		To this end we consider the homomorphism
		\begin{equation}\label{E:eee}
			\Diff_c(M;N)\to\Aut(E),\qquad g\mapsto Tg|_N,
		\end{equation}
		where $\Aut(E)$ denotes the Lie group of all fiberwise linear diffeomorphisms of the normal bundle $E=TM|_N/TN$.
		After disregarding the connected components of $\Aut(E)$ which are not in the image, the homomorphism in~\eqref{E:eee} becomes a principal bundle with structure group $\{g\in\Diff_c(M,N)\mid\forall x\in N:T_xg=\id\}$, a splitting Lie subgroup in $\Diff_c(M,N)$.
		As $\Diff_c(M,N,\gamma)$ is the preimage of the group $\Aut(E,\gamma)$ under the map in \eqref{E:eee}, it suffices to show that $\Aut(E,\gamma)$ is a splitting Lie subgroup in $\Aut(E)$.
		The homomorphism
		\begin{equation}\label{E:eeee}
			\Aut(E)\to\Diff(N)
		\end{equation}
		is a principal $\Gamma(\GL(E))$ bundle, after disregarding the connected components of $\Diff(N)$ which are not in the image.
		If $U\subseteq\Diff(N)$ is an open neighborhood of the identity and $s\colon U\to\Aut(E)$ is a local section of \eqref{E:eeee} with $s(\id)=\id$, 
		then 
		\begin{equation}\label{E:eeeee}
			U\times\Gamma(\GL(E))\to\Aut(E)|_U,\quad(f,\phi)\mapsto s(f)\phi
		\end{equation} 
		is a diffeomorphism onto an open neighborhood of the identity in $\Aut(E)$.
		It is easy to see that the local section $s$ may be chosen to take values in $\Aut(E,\gamma)$.
		Then the diffeomorphism in \eqref{E:eeeee} maps $U\times\Gamma(\GL(E,\gamma))$ onto $\Aut(E,\gamma)|_U$.
		This shows that $\Aut(E,\gamma)$ is a splitting Lie subgroup in $\Aut(E)$, for $\Gamma(\GL(E,\gamma))$ is a splitting Lie subgroup in $\Gamma(\GL(E))$. 
	\end{proof}
	
	Suppose $N\in\Gr(M)$ and $\gamma\in\Gamma\bigl((|\Lambda|_N\otimes\Ann TN)\setminus N\bigr)$.
	Combining Propositions~\ref{P:DiffcM.good} and \ref{P:aug.full.iso} with Lemma~\ref{L:aug2}, 
	we obtain a Fr\'echet manifold structure on the augmented Grassmannian $\G=G\cdot(N,\gamma)$ 
	such that the $G$ action is smooth and admits local smooth sections.
	Moreover, $\mathcal G$ is open and closed in 
	\begin{equation}\label{E:fff}
		\Gr^{\aug,\ann}(M)=\bigl\{(\tilde N,\tilde\ga)\in\Gr^{\aug}(M):\tilde\ga\in\Ga((|\Lambda|_{\tilde N}\otimes\Ann T\tilde N)\setminus\tilde N)\bigr\}.
	\end{equation}
	The latter is a splitting smooth subbundle of $\Gr^{\aug}(M)\to\Gr(M)$ in view of the trivializations in~\eqref{E:triv.aug} and the splitting smooth submanifold in~\eqref{E:Ann}.
	According to Lemma~\ref{L:aug3} the orbit map $G\to\mathcal G$, $g\mapsto g\cdot(N,\gamma)$ is a principal $G_{(N,\gamma)}$ bundle.

	Now, for any $(N,\gamma)\in\Gr^{\aug,\ann}(M)$ the coadjoint orbit of $J(N,\ga)$ is a good coadjoint orbit of $\Diff_c(M)$, cf.~\eqref{E:J.Diffc}.
	More precisely, using Lemma~\ref{lemac} we obtain:

	\begin{theorem}[{\cite[Theorem~2.2]{HV24}}]\label{T:full}
		Suppose $(N,\gamma)\in\Gr^{\aug,\ann}(M)$ 
		and consider the augmented Grassmannian $\mathcal G=\Diff_c(M)\cdot(N,\gamma)$, open and closed in $\Gr^{\aug,\ann}(M)$.
		Then its image $J(\G)$ under the map $J$ in \eqref{E:J.Diffc} is a good coadjoint orbit of $\Diff_c(M)$ 
		and the (formal) pull back of the KKS symplectic form, denoted by $\om=J^*\om_{\KKS}$,
		is a $\Diff_c(M)$ equivariant smooth (weakly) symplectic form on $\G$, characterized by
		\[
			\omega_{(\tilde N,\tilde\gamma)}(\zeta_X,\zeta_Y)=\int_{\tilde N}\tilde\gamma([X,Y])
		\]
		for all $(\tilde N,\tilde\gamma)\in\G$ and $X,Y\in\mathfrak X_c(M)$.
		Furthermore, $J$ restricts to an (equivariant) moment map for the (Hamiltonian) action of $\Diff_c(M)$ on $\G$.
	\end{theorem}

	By Remark~\ref{R:unique}, the smooth structure on $\G$ must coincide with the one considered in \cite[Section~2.3]{HV24}. 
	Hence, $\G$ is a union of connected components in the open subset of nowhere zero elements in the regular cotangent bundle $T^*_{\reg}\Gr(M)$,
	and the pullback of the KKS form is the cotangent bundle symplectic form.
	The symplectic manifold $(\G,\om)$ can be obtained by symplectic reduction at zero
	on the right leg of  the Holm--Marsden EPDiff dual pair \cite{HM05,GBV12}, as shown in \cite[Theorem~2.3]{HV24}.
	This leg is the moment map for the $\Diff(S)$ action on (an open subset of) the regular cotangent bundle of $\Emb(S,M)$.

\subsubsection*{Augmentations with nowhere zero pullback}
Augmented submanifolds $(N,\ga)$ of $M$, with nowhere zero 1-form density $\io_N^*\ga\in\Om^1(N,|\La|_N)$,
can also parametrize coadjoint orbits of $G=\Diff_c(M)$.
Smooth structures on such $G$ orbits are available for 1-dimensional submanifolds $N$ of  $M$. 
In this case, a nowhere zero element in $ \Om^1(N,|\La|_N)$ is the same as a Riemannian metric plus an orientation on $N$,
and its $\Diff(N)$ orbit is characterized by the lengths of the connected components of $N$.

\begin{proposition}\label{P:aug.dens}
	Let $G=\Diff_c(M)$. 
	Suppose $N\in\Gr_{\dim=1}(M)$ and $\ga\in\A(N)$ are such that $\io_N^*\ga\in\Om^1(N,|\Lambda|_N)$ is nowhere zero.
	Then $G_N$ has good orbit at $\gamma$ and the orbit is open and closed in 
	\begin{equation}\label{E:e}
		\{\tilde\gamma\in\Gamma(|\Lambda|_N\otimes T^*M|_N):\io_N^*\tilde\ga\cong\io_N^*\ga\},
	\end{equation}	
	a splitting smooth submanifold of $\mathfrak A(N)$.
\end{proposition}

\begin{proof}
	Put $\rho:=\iota_N^*\gamma$.
	According to \cite[Example~4.16]{HV22}, the orbit $\Diff(N)\cdot\rho$ is a splitting smooth submanifold in $\Omega^1(N;|\Lambda|_N)$ and the $\Diff(N)$ action on this orbit admits local smooth sections. 
	Recall that the homomorphism $\Diff_c(M,N)\to\Diff(N)$ admits a local smooth right inverse at the identity.
	We conclude that $\Diff_c(M,N)\cdot\rho$ is open and closed in $\Diff(N)\cdot\rho$, hence a splitting smooth submanifold in $\Omega^1(N;|\Lambda|_N)$, and the $\Diff_c(M,N)$ action on this orbit admits local smooth sections too.
	In particular, there exists an open neighborhood $V$ of $\rho$ in $\Diff_c(M,N)\cdot\rho$ and a smooth map $v\colon V\to\Diff_c(M,N)$ such that $v(\tilde\rho)\cdot\rho=\tilde\rho$ for all $\tilde\rho\in V$.

	Let $L$ denote the preimage of $\Diff_c(M,N)\cdot\rho$ under the linear map
	\begin{equation}\label{E:ioN}
		\iota_N^*\colon\Gamma\bigl(|\Lambda|_N\otimes T^*M|_N\bigr)\to\Omega^1(N;|\Lambda|_N).
	\end{equation}
	Since this map admits a bounded linear right inverse, it may be considered as fiber bundle whose fibers are affine spaces over $\ker(\iota_N^*)=\Gamma(|\Lambda|_N\otimes\Ann TN)$.
	In particular, $L$ is a splitting smooth submanifold in $\Gamma(|\Lambda|_N\otimes T^*M|_N)$ which is stable under $\Diff_c(M,N)$ and contains $\gamma$.
	It remains to show that the $\Diff_c(M,N)$ action on $L$ admits local smooth sections.
	It suffices to construct a local smooth right inverse in a neighborhood of $\gamma$.
	To this end, let 
	\[
		F=(\iota_N^*)^{-1}(\rho)=\gamma+\Gamma(|\Lambda|_N\otimes\Ann TN)
	\]
	denote the fiber of \eqref{E:ioN} over $\rho$.
	Choose $Y\in\mathfrak X_c(M)$ such that $Y$ is tangential to $N$ 
	and such that $\nu=\gamma(Y)$ is a (nowhere vanishing) volume density on $N$.
	For each $\tilde\gamma\in F$ there exists $Z_{\tilde\gamma}\in\mathfrak X_c(M)$ depending smoothly on $\tilde\gamma$ such that 
	$Z_{\tilde\gamma}|_N=0$ and $D_xZ_{\tilde\gamma}=Y_x\otimes\frac{\tilde\gamma_x-\gamma_x}{\nu_x}\in\operatorname{End}(T_xM)$ 
	for all $x\in N$.
	Let $g_{\tilde\gamma}\in\Diff_c(M)$ denote the flow of $Z_{\tilde\gamma}$ at time one.
	Clearly, $g_{\tilde\gamma}|_N=\id$.
	Since $D_xZ_{\tilde\gamma}$ has strict upper triangular form with respect to the decomposition $T_xM=T_xN\oplus\ker\gamma_x$, 
	we obtain $T_xg_{\tilde\gamma}=\id+Y_x\otimes\frac{\tilde\gamma_x-\gamma_x}{\nu_x}$ 
	and thus $(T_xg_{\tilde\gamma})^*\frac{\gamma_x}{\nu_x}=\frac{\tilde\gamma_x}{\nu_x}$ for all $x\in N$.
	We conclude that $w\colon F\to\Diff_c(M,\iota_N)$, $w(\tilde\gamma)=g_{\tilde\gamma}^{-1}$
	is a global smooth right inverse for the $\Diff_c(M,\iota_N)$ action on $F$.
	More explicitly, $w(\tilde\gamma)\cdot\gamma=\tilde\gamma$ for all $\tilde\gamma\in F$.

	If $\tilde\gamma\in L$ and $\iota_N^*\tilde\gamma\in V$, then $v(\iota_N^*\tilde\gamma)^{-1}\cdot\tilde\gamma\in F$.
	Hence, $U=(\iota_N^*)^{-1}(V)$ is an open neighborhood of $\gamma$ in $L$ and the smooth map
	\[
		u\colon U\to\Diff_c(M,N),\qquad u(\tilde\gamma)=v(\iota_N^*\tilde\gamma)w\bigl(v(\iota_N^*\tilde\gamma)^{-1}\cdot\tilde\gamma\bigr)
	\]
	satisfies $u(\tilde\gamma)\cdot\gamma=\tilde\gamma$ for all $\tilde\gamma\in U$.
	Hence, $u$ is the desired local right inverse.
\end{proof}

Suppose $\dim S=1$ and $\rho\in\Omega^1(S;|\Lambda|_S)$ is nowhere vanishing.
Then
\[
	\Gr^{\aug}_{S,\rho}(M)=\bigl\{(N,\ga)\in\Gr^{\aug}(M):(N,\io_N^*\ga)\cong(S,\rho)\bigr\}
\]
is a splitting smooth subbundle of $\Gr^{\aug}_S(M)\to\Gr_S(M)$ in view of the trivializations in \eqref{E:triv.aug} and the fact that \eqref{E:e} is a splitting submanifold of $\mathfrak A(N)$.
Let $(N,\gamma)\in\Gr^{\aug}_{S,\rho}(M)$.
Combining Propositions~\ref{P:DiffcM.good} and \ref{P:aug.dens} with Lemma~\ref{L:aug2}, 
we obtain a Fr\'echet manifold structure on the augmented Grassmannian $\G=G\cdot(N,\gamma)$ 
such that the $G$ action is smooth and admits local smooth sections.
Moreover, $\mathcal G$ is open and closed in $\Gr^{\aug}_{S,\rho}(M)$.
Now, for every $(N,\gamma)\in\Gr^{\aug}_{S,\rho}(M)$ the coadjoint orbit of $J(N,\ga)$ is a good coadjoint orbit of $\Diff_c(M)$, cf.~\eqref{E:J.Diffc}.
More precisely, Lemma~\ref{lemac} yields:

\begin{theorem}[{\cite[Theorem~2.3]{HV24}}]\label{T:non0}
	Suppose $\dim S=1$, assume $\rho\in\Omega^1(S;|\Lambda|_S)$ is nowhere vanishing, $(N,\gamma)\in\Gr^{\aug}_{S,\rho}(M)$ 
	and consider the augmented Grassmannian $\G=\Diff_c(M)\cdot(N,\ga)$, open and closed $\Gr^{\aug}_{S,\rho}(M)$. 
	Then its image $J(\G)$ under the map $J$ in~\eqref{E:J.Diffc} is a good coadjoint orbit of $\Diff_c(M)$, 
	and the (formal) pullback  of the KKS symplectic form, 
	denoted by $\om=J^*\om_{\KKS}$,
	is a $\Diff_c(M)$ equivariant smooth (weakly) symplectic form on  $\G$ characterized by
	\[
		\omega_{(\tilde N,\tilde\gamma)}(\zeta_X,\zeta_Y)=\int_{\tilde N}\tilde\gamma([X,Y]),
	\]
	for all $(\tilde N,\tilde\gamma)\in\G$ and $X,Y\in\mathfrak X_c(M)$.
	Furthermore, $J$ restricts to an (equivariant) moment map for the (Hamiltonian) action of $\Diff_c(M)$ on $\G$.
\end{theorem}

The symplectic manifold $(\G,\om)$ can be obtained by reduction at a nowhere zero element of $\Om^1(S,|\La|_S)$,
for the Hamiltonian $\Diff(S)$ action on (an open subset of) the regular cotangent bundle of $\Emb(S,M)$,
whose moment map is the right leg of the Holm--Marsden EPDiff dual pair \cite[Theorem~2.3]{HV24}.


\subsection{The group of contact diffeomorphisms}
	Suppose $\xi\subseteq TM$ is a contact hyperplane distribution on $M$ and let $L:=TM/\xi$ denote the contact  line bundle.
	The canonical action of the contact diffeomorphism group $G=\Diff_c(M,\xi)$ on $T^*M$ restricts to an action on the dual line bundle $L^*\subseteq T^*M$.
	For $N\in\Gr(M)$ we let
	\[
		\mathfrak A(N):=\Gamma\bigl((|\Lambda|_N\otimes L|^*_N)\setminus N\bigr)
	\]
	denote the space of all nowhere vanishing smooth sections of the vector bundle $|\Lambda|_N\otimes L|^*_N$.
	We equip $\mathfrak A(N)$ with the Fr\'echet manifold structure it inherits as an open subset in the Fr\'echet space of all sections with the $C^\infty$ topology.
	For $g\in G$ we let $\mathfrak A(g)\colon\mathfrak A(N)\to\mathfrak A(g\cdot N)$ denote the natural map induced by $g$.
	Since this is the restriction of a bounded linear map between Fr\'echet spaces, $\mathfrak A(g)$ is smooth.
	Clearly, $\mathfrak A$ defines a functor from $\cataug(M,G)$ into the category of Fr\'echet manifolds.

	The augmented Grassmannian associated to this functor $\mathfrak A$ will be denoted by $\Gr^{\aug}(M,\xi)$.
	One readily shows \cite[Proposition~4.2(a)]{HV24} that the map
	\begin{equation}\label{E:J.contact}
		J\colon\Gr^{\aug}(M,\xi)\to\mathfrak g^*,\qquad\langle J(N,\gamma),X\rangle=\int_N\gamma(X),\quad X\in\mathfrak g
	\end{equation}
	is injective and $\Diff(M,\xi)$ equivariant.
	Here $\mathfrak g=\mathfrak X_c(M,\xi)$ denotes the Lie algebra of compactly supported contact vector fields on $M$.
	
	\begin{proposition}\label{P:aug.contact.orbit}
		$\mathfrak A$ is an augmentation functor for the $G$ action on $M$.
	\end{proposition}

	\begin{proof}
		It remains to check the smoothness assumption.
		To this end consider $N\in\Gr(M)$ and suppose $g\colon Z\to G_N$ is smooth into $G$.
		Since $\Diff_c(M,N)$ is an (initial) submanifold of $\Diff_c(M)$, cf.~\cite[Lemma~2.1(b)]{HV20}, the map $g$ is smooth into $\Diff_c(M,N)$ too.
		Choose a fiberwise linear projection $\varpi\colon T^*M|_N\to L|_N^*$ and let $\iota\colon L^*|_N\to T^*M|_N$ denote the inclusion.
		The map 
		\begin{align*}
			\Diff_c(M,N)\times\Gamma(|\Lambda|_N\otimes L|_N^*)&\to\Gamma(|\Lambda|_N\otimes L|_N^*)
			\\
			(f,\phi)&\mapsto\Bigl(|\Lambda|(f|_N)\otimes\bigl(\varpi\circ(T^*f)|_N\circ\iota\bigr)\Bigr)\circ\phi\circ f^{-1}
		\end{align*}
		is smooth and restricts to the $G_N$ action on $\mathfrak A(N)$.
		Hence, $Z\times\mathfrak A(N)\to\mathfrak A(N)$, $(z,\gamma)\mapsto\mathfrak A(g(z))\cdot\gamma$ is smooth.
	\end{proof}

	Combining Proposition~\ref{P:aug.contact.orbit} with Lemma~\ref{L:aug1} we get a smooth structure on $\Gr^{\aug}(M,\xi)|_{G\cdot N}$ for each $N\in\Gr^{\iso}(M,\xi)$.
	Since $G\cdot N$ is open and closed in $\Gr^{\iso}(M,\xi)$ this yields a smooth structure on the disjoint union, $\Gr^{\aug,\iso}(M,\xi)$, such that the forgetful map $\Gr^{\aug,\iso}(M,\xi)\to\Gr^{\iso}(M,\xi)$ is a smooth fiber bundle.

	Form \cite[Lemma~3.4]{HV22} we obtain, cf.~\cite[Theorem~4.10(a)]{HV22}:

	\begin{proposition}\label{P:aug.contact.iso}
		Let $G=\Diff_c(M,\xi)$ and $N\in\Gr^{\iso}(M,\xi)$.
		Then $G_N$ has good (open) orbit at each $\gamma\in\mathfrak A(N)$.
	\end{proposition}

	Suppose $(N,\gamma)\in\Gr^{\aug,\iso}(M,\xi)$.
	Combining Propositions~\ref{P:contact.good} and \ref{P:aug.contact.iso} with Lemma~\ref{L:aug2},
	we obtain a Fr\'echet manifold structure on the augmented Grassmannian $\G=G\cdot(N,\gamma)$,
	such that the $G$ action is smooth and admits local smooth sections.
	Moreover, $\mathcal G$ is open and closed in $\Gr^{\aug,\iso}(M,\xi)$.

	Now, for any $(N,\gamma)\in\Gr^{\aug,\iso}(M,\xi)$ the coadjoint orbit of $J(N,\ga)$ is a good coadjoint orbit of $\Diff_c(M,\xi)$, cf.~\eqref{E:J.contact}.
	More precisely, using Lemma~\ref{lemac} we obtain:

	\begin{theorem}[{\cite[Theorem~4.10]{HV22}}]\label{T:contact}
		Let $(M,\xi)$ be a contact manifold, suppose $(N,\gamma)\in\Gr^{\aug,\iso}(M,\xi)$ 
		and consider the augmented nonlinear Grassmannian $\G=\Diff_c(M,\xi)\cdot(N,\gamma)$, open and closed in $\Gr^{\aug,\iso}(M,\xi)$.
		Then its image $J(\G)$ under the map $J$ in~\eqref{E:J.contact} is a good coadjoint orbit of $\Diff_c(M,\xi)$ 
		and the (formal) pullback of the KKS symplectic form, denoted by $\omega=J^*\om_{\KKS}$, 
		is a $\Diff_c(M,\xi)$ equivariant smooth (weakly) symplectic form on $\G$, characterized by
		\[
			\omega_{(\tilde N,\tilde\gamma)}(\zeta_X,\zeta_Y)=\int_{\tilde N}\tilde\gamma([X,Y])
		\]
		for all $(\tilde N,\tilde\gamma)\in\G$ and $X,Y\in\mathfrak X_c(M,\xi)$.
		Furthermore, $J$ restricts to an (equivariant) moment map for the (Hamiltonian) action of $\Diff_c(M,\xi)$ on $\G$.
	\end{theorem}
	
	By Remark~\ref{R:unique} the smooth structure on $\G$ coincides with the one considered in \cite[Section~4.3]{HV22}.
	The symplectic manifold $(\G,\om)$ can be obtained by symplectic reduction at zero
	on the right leg of the EPContact dual pair \cite[Theorem~4.12]{HV22}. This leg is
	the moment map for the $\Diff(S)$ action on the space of fiberwise linear embeddings $|\La|^*_S\to L^*$.

	We expect that $G$ also has good isotropy at $(N,\gamma)\in\Gr^{\aug,\iso}(M,\xi)$, cf.~\cite[Remark~4.12]{HV22}.


\subsection{The group of exact volume preserving diffeomorphisms}\label{S:diffvol}
Next we present  coadjoint orbits of the group $G=\Diff_{c,\ex}(M)$, that can be parametrized by augmented Grassmannians. 
Coadjoint orbits of $G$ parametrized by decorated Grassmannians already appeared in Section \ref{S:isovol}.
We know, by Proposition~\ref{P:aug.full.orbit}, that 
\[
		\mathfrak A(N):=\Gamma\bigl((|\Lambda|_N\otimes T^*M|_N)\setminus N\bigr), \quad N\in\Gr(M),
\]
is an augmentation functor for the $\Diff_c(M)$ action. By restriction, $\A$ is an augmentation functor also for the $G$ action.

\subsubsection*{Codimension at least two}
In this case, \cite[Lemma~3.1]{HV24} shows that the infinitesimal orbits of $\Diff_c(M)$ and $G=\Diff_{c,\ex}(M)$ in $\Gr_{\codim\geq2}^{\aug}(M)$ coincide.
Correspondingly, we have the following analogue of Proposition~\ref{P:aug.full.iso}:

\begin{proposition}\label{P:Diff.ex.good.O}
	Let $G=\Diff_{c,\ex}(M)$.
	If $N\in\Gr_{\codim\geq2}(M)$ and $\gamma\in\Gamma\bigl((|\Lambda|_N\otimes\Ann TN)\setminus N\bigr)$, then $G_N$ has good orbit at $\gamma$.
	Moreover, this orbit is open and closed in $\Ga((|\Lambda|_N\otimes\Ann TN)\setminus N)$.
\end{proposition}

\begin{proof}
	The action of $\Diff_{c,\ex}(M;\iota_N)$ on $\Gamma(\Ann TN\setminus N)$ factors through
	\[
		\Diff_{c,\ex}(M;\iota_N)\to\Gamma\bigl(\SAut(TM|_N/TN)\bigr),\qquad g\mapsto Tg|_N.
	\]
	Here $\SAut(TM|_N/TN)$ denotes the bundle of fiberwise linear isomorphisms with determinant one.
	Using a tubular neighborhood of $N$ in $M$, one can show that this homomorphism admits a local smooth right inverse near the identity.
	Hence, denoting the normal bundle by $E=TM|_N/TN$, 
	it suffices to show that the action of $\Gamma(\SAut(E))$ on $\Gamma(E^*\setminus N)$ admits local smooth sections.
	To this end, we proceed as in Proposition~\ref{P:aug.full.iso}, 
	except that we modify the definition of the vector bundle homomorphism $A_\beta\colon E\to E$ in order to get $\det(A_\beta)=1$.
	More precisely, given $\alpha,\beta\in\Gamma(E^*)$ with $\alpha(\sharp\beta)\neq0$, we define $A_\beta$ by 
	$A_\beta(\sharp\alpha)=\sharp\beta$ and $A_\beta|_{\ker\alpha}=\bigl(\frac{\alpha(\sharp\beta)}{|\alpha|^2}\bigr)^{\frac1{1-\rk(E)}}\cdot\id$, where $\rk E\ge 2$ by assumption.
	Then $u(\beta)=(A_\beta^*)^{-1}$ is a local smooth section near $\alpha$.
\end{proof}

	Suppose $(N,\gamma)\in\Gr^{\aug,\ann}_{\codim\geq2}(M)$.
	Combining Propositions~\ref{trans} and \ref{P:Diff.ex.good.O} with Lemma~\ref{L:aug2},
	we obtain a Fr\'echet manifold structure on the augmented Grassmannian $\G_0=G\cdot(N,\gamma)$
        such that the $G$ action on $\mathcal G_0$ is smooth and admits local smooth sections.
	Moreover, $\mathcal G$ is open and closed in $\Gr^{\aug,\ann}_{\codim\geq2}(M)$, hence a connected component, as $G$ is connected.

	Let $\mathfrak g=\mathfrak X_{c,\ex}(M)$ denote the Lie algebra of compactly supported exact divergence free vector fields on $M$.
	We have an injective and $\Diff(M,\mu)$ equivariant map, similarly to \eqref{E:J.Diffc}:
	\begin{equation}\label{E:J.Diffc2}
		J\colon\Gr^{\aug}(M)\to\mathfrak g^*,\qquad\langle J(N,\gamma),X\rangle=\int_N\gamma(X),\quad X\in\mathfrak g.
	\end{equation}
	Now, for any $(N,\gamma)\in\Gr^{\aug,\ann}_{\codim\geq2}(M)$ the coadjoint orbit of $J(N,\ga)$ is a good coadjoint orbit of  $\Diff_{c,\ex}(M)$.
	More precisely, using Lemma~\ref{lemac} we obtain:

	\begin{theorem}[{\cite[Corollary~3.4]{HV24}}]\label{T:cex}
		Let $\mu$ be volume form on $M$ and consider a connected component $\G_0$ in the augmented nonlinear Grassmannian $\Gr^{\aug,\ann}_{\codim\geq2}(M)$.
		Then its image $J(\G)$ under the map $J$ in~\eqref{E:J.Diffc2} is a good coadjoint orbit of $\Diff_{c,\ex}(M)$
		and the (formal) pullback of the KKS symplectic form, denoted by $\om=J^*\om_{\KKS}$,  
		is a $\Diff_{c,\ex}(M)$ equivariant smooth (weakly) symplectic form on $\G_0$, characterized by
		\[
			\omega_{(N,\gamma)}(\zeta_X,\zeta_Y)=\int_N\gamma([X,Y])
		\]
		for all $(N,\gamma)\in\G_0$ and $X,Y\in\mathfrak X_{c,\ex}(M)$.
		Furthermore, $J$ restricts to an (equivariant) moment map for the (Hamiltonian) action of $\Diff_{c,\ex}(M)$ on $\G_0$.
	\end{theorem}

	The symplectic manifold $(\G_0,\om)$ is a connected component of the (regular) cotangent bundle to $\Gr(M)$.
	It can be obtained also by symplectic reduction at zero
	on the right leg of the EPDiffvol dual pair \cite[Corollary~3.4]{HV24},
	the moment map for the $\Diff(S)$ action on (an open subset of) the regular cotangent bundle of $\Emb(S,M)$.

\subsubsection*{Codimension one}
	In the codimension one case, we need the isodrastic distribution on $\Gr(M)$ from Section \ref{S:isovol}:
	\[
		D_N:=\{X\in\Ga(TM|_N/TN):\io_N^* i_X\mu\text{ exact}\}.
	\]
	By Proposition~\ref{L:Fcex.int}, its leaves coincide with the $G=\Diff_{c,\ex}(M)$ orbits and the $G$ action on each isodrastic leaf admits local smooth sections.	
	Hence, in order to parametrize coadjoint orbits of $G$, we change the augmentation to the quotient
	\[
		\A(N):=\Gamma(|\Lambda|_N\otimes T^*M|_N)/\Ann_{\reg}\D_N, \quad N\in\Gr(M),
	\]
	where $\Ann_{\reg}\D_N\subseteq\Ga(|\La|_N\otimes\Ann (TN))$ is the annihilator of $\D_N$,
	and we let $\Gr^{[\aug]}(M)$ denote the corresponding augmented nonlinear Grassmannian.
	Then the map $J$ in~\eqref{E:J.Diffc} factors to an injective $\Diff(M,\mu)$ equivariant map
	\begin{equation}\label{E:J.aaug}
		\bar J\colon\Gr^{[\aug]}(M)\to\g^*,\qquad\langle \bar J(N,[\gamma]),X\rangle=\int_N\gamma(X).
	\end{equation}
	where $\mathfrak g=\mathfrak X_{c,\ex}(M)$.
	The isomorphism \eqref{E:mu}  provides an isomorphism of finite dimensional vector spaces,
		\begin{equation}\label{E:mu.Gr2}
			\mu_N:H^0(N;\mathfrak o_N)\to\Ann_{\reg}\mathcal D_{N}.
		\end{equation}
	Hence, $\dim \Ann_{\reg}\D_N$ is the number of orientable connected components of $N$.
	The pullback map factors to $ \A(N)\to\Om^1(N,|\La|_N)$, $[\ga]\mapsto\io_N^*\ga$.
	
	We know that $G$ has good orbit at $N\in\Gr(M)$ (by Proposition \ref{L:Fcex.int}),
	and the orbit $G\cdot N$ is the isodrastic leaf through $N$.
	Thus, with Lemma \ref{L:aug1}, we get a Fr\'echet manifold structure on the augmented Grassmannian
	$\Gr^{[\aug]}(M)|_{G\cdot N}$ such that the $G$ equivariant forgetful map to
	the isodrast $G\cdot N$ is a smooth fiber bundle with typical fiber $ \A(N)$.
	Moreover, the $G$ action on this augmented Grassmannian is smooth.
	
To get smooth structures on the orbits, we have to restrict to 1-dimensional submanifolds, thus $\dim M=2$ and $G=\Ham_c(M)$.
Then the following analogue of Proposition~\ref{P:aug.dens} holds:

\begin{proposition}\label{P:undoi}
	Let $G=\Ham_c(M)$ with $\dim M=2$. 
	If $N\in\Gr_{\dim=1}(M)$ 	and if $[\gamma]\in \A(N)$ such that $\io_N^*\ga\in\Om^1(N,|\Lambda|_N)$ is nowhere zero, 
	then $G_N$ has good orbit at $[\gamma]$. Moreover,  the orbit is open and closed in 
	\begin{equation}\label{E:ee}
		\bigl\{[\tilde\gamma]\in\A(N):\io_N^*\tilde\ga\cong\io_N^*\ga\bigr\},
	\end{equation}
	a splitting smooth submanifold of $\A(N)$.
\end{proposition}

\begin{proof}
	From Proposition~\ref{pp} we know that the natural homomorphism $G_N\to\Diff(N)$ admits a smooth local right inverse near the identity.
        Hence, proceeding exactly as in the proof of Proposition~\ref{P:aug.dens}, it remains to show that the action of $G_{\iota_N}$ on
        \[
                F=\gamma+\Gamma\bigl(|\Lambda|_N\otimes\Ann TN\bigr)
        \]
	admits local smooth sections.
	W.l.o.g.\ we may assume $N=S^1$, $M=S^1\times\mathbb R$, and
	$\mu=d\theta\wedge dt$ where $\theta$ and $t$ denote the projections onto $S^1$ and $\mathbb R$, respectively.
	Consider the volume density $\nu=\gamma(\partial_\theta)\in\Gamma(|\Lambda|_N)$.
	Then any $\tilde\gamma\in F$  is of the form $\tilde\gamma=\gamma+\nu\otimes f_{\tilde\gamma}dt$ with
	$f_{\tilde\gamma}=\frac{(\tilde\gamma-\gamma)(\partial_t)}\nu\in C^\infty(N)$.
	Consider the compactly supported Hamilton function $h_{\tilde\gamma}(\theta,t)=\frac12t^2f_{\tilde\gamma}(\theta)\lambda(t)$
	on $M$ where $\lambda(t)$ is a compactly supported smooth function that equals $1$ in a neighborhood of $t=0$.
	The corresponding Hamilton vector field $X_{\tilde\gamma}\in\ham_c(M)$ vanishes along $N$ and for its derivative
	we find $D_{(\theta,0)}X_{\tilde\gamma}=f_{\tilde\gamma}(\theta)\partial_\theta\otimes dt$.
	Let $g_{\tilde\gamma}\in\Ham_c(M)$ denote the flow of $X_{\tilde\gamma}$ at time one.
	Then $g_{\tilde\gamma}|_N=\id$ and $T_{(\theta,0)}g_{\tilde\gamma}=\id+f_{\tilde\gamma}(\theta)\partial_\theta\otimes dt$.
	Hence, $g_{\tilde\gamma}^*\frac\gamma\nu=\frac{\tilde\gamma}\nu$.
	We conclude that $F\to G_{\iota_N}$, $\tilde\gamma\mapsto g_{\tilde\gamma}^{-1}$ is a global smooth right inverse for the $G_{\iota_N}$ action on $F$.
\end{proof}

Suppose $\dim S=1$ and assume $\rho\in\Omega^1(S;|\Lambda|_S)$ is nowhere vanishing.
Moreover, let $\mathcal L$ denote an isodrastic leaf in $\Gr_S(M)$, cf.~Proposition~\ref{L:Fcex.int}.
Then
\[
	\Gr^{[\aug]}_{S,\rho}(M)|_{\mathcal L}=\bigl\{(N,[\ga])\in\Gr^{[\aug]}(M)|_{\mathcal L}:(N,\io_N^*\ga)\cong(S,\rho)\bigr\}
\]
is a splitting smooth subbundle of $\Gr^{[\aug]}(M)|_{\mathcal L}\to\mathcal L$ in view of the trivializations in \eqref{E:triv.aug} and the fact that \eqref{E:ee} is a splitting submanifold of $\mathfrak A(N)$.
Let $(N,[\gamma])\in\Gr^{[\aug]}_{S,\rho}(M)|_{\mathcal L}$.
Combining Propositions~\ref{L:Fcex.int} and \ref{P:undoi} with Lemma~\ref{L:aug2}, 
we obtain a Fr\'echet manifold structure on the augmented Grassmannian $\bar\G:=G\cdot(N,[\gamma])$ 
such that the $G$ action is smooth and admits local smooth sections.
Moreover, $\bar\G$ is open and closed in $\Gr^{[\aug]}_{S,\rho}(M)|_{\mathcal L}$.

Now, for any 1-dimensional submanifold $N$ of the 2-dimensional manifold $M$, endowed with 
$[\gamma]\in\Gamma(\Lambda|_N\otimes T^*M|_N)/\Ann_{\reg}\D_N$ such that $\io_N^*\ga\in\Om^1(N,|\La|_N)$ nowhere zero, 
the coadjoint orbit of $\bar J(N,[\ga])$ is a good coadjoint orbit of  $\Ham_c(M)$.
More precisely, using Lemma~\ref{lemac} we obtain:

\begin{theorem}[{\cite[Corollary~3.4]{HV24}}]\label{T:[]}
	Let $M$ be a 2-dimensional symplectic manifold, 
	suppose $\mathcal L$ is an isodrastic leaf in $\Gr_S(M)$ where $\dim S=1$, 
	assume $\rho\in\Omega^1(S;|\Lambda|_S)$ is nowhere vanishing, 
	let $(N,[\gamma])\in\Gr^{[\aug]}_{S,\rho}(M)|_{\mathcal L}$,
	and consider the augmented Grassmannian $\bar\G=\Ham_c(M)\cdot(N,[\gamma])$, open and closed in $\Gr^{[\aug]}_{S,\rho}(M)|_{\mathcal L}$.
	Then its image $\bar J(\bar\G)$ under the map $\bar J$ in~\eqref{E:J.aaug} is a good coadjoint orbit of $\Ham_c(M)$
	and the (formal) pullback of the KKS symplectic form, denoted by $\om=\bar J^*\om_{\KKS}$,  
	is a $\Ham_c(M)$ equivariant smooth (weakly) symplectic form on $\bar\G$, characterized by
	\[
		\omega_{(\tilde N,[\tilde\gamma])}(\zeta_X,\zeta_Y)=\int_{\tilde N}\tilde\gamma([X,Y])
	\]
	for all $(\tilde N,[\tilde\gamma])\in\bar\G$ and $X,Y\in\ham_c(M)$.
	Furthermore, $\bar J$ restricts to an (equivariant) moment map for the (Hamiltonian) action of $\Ham_c(M)$ on $\bar\G$.
\end{theorem}
	
The symplectic manifold $(\bar\G,\om)$ can be obtained by symplectic reduction  at a nowhere zero element of $\Om^1(S,|\La|_S)$
on the right leg of the (second) EPDiffvol dual pair \cite[Corollary~6.1]{HV24},
the moment map for the $\Diff(S)$ action on (an open subset of) the regular cotangent bundle of a disjoint union of isodrastic leaves in $\Emb(S,M)$.
	

\subsection{Decorations as augmentations}
By composition with the canonical functor $\cataug(M,G)\to\catclosed$  in \eqref{E:cats}, every decoration functor gives rise to an augmentation functor for the $G$ action $M$.
The smooth structures associated with these two functors are compatible in the following sense:

\begin{lemma}\label{L:deco.aug}
	Suppose $\mathfrak D$ is a decoration functor and suppose $G$ is a Lie group acting smoothly on a finite dimensional manifold $M$.
	Let $\mathfrak A$ denote the functor obtained by composing $\mathfrak D$ with the functor in \eqref{E:cats}.
	Then $\mathfrak A$ is an augmentation functor for the $G$ action on $M$.
	Moreover, for $\mu\in\mathfrak D(S)$ and $(N,\nu)\in\Gr_{S,\mu}^{\deco}(M)$ the following hold true:
	\begin{enumerate}[(a)]
		\item	If $G$ has good orbit at $N$, then the canonical $G$ equivariant identification
			\begin{equation}\label{E:aug.deco}
				\Gr^{\aug}(M)|_{G\cdot N}=\Gr^{\deco}_S(M)|_{G\cdot N}
			\end{equation}
			is a diffeomorphism of fiber bundles over $G\cdot N$ with typical fiber $\mathfrak D(S)$.
		\item	If $\Diff(S)$ has good orbit at $\mu$ and $G$ has good extension at $N$, then $G_N$ has good orbit at $\nu$.
			Moreover, $G_N\cdot\nu$ is open and closed in $\Diff(N)\cdot\nu$.
		\item	If $\Diff(S)$ has good orbit at $\mu$, and $G$ has good extension and orbit at $N$, then the canonical $G$ equivariant inclusion
			\begin{equation}\label{E:ggg}
				G\cdot(N,\nu)\subseteq\Gr_{S,\mu}^{\deco}(M)|_{G\cdot N}
			\end{equation}
			is a diffeomorphism onto an open and closed subset.
		\item	Suppose $\Diff(S)$ has good isotropy at $\mu$, $G$ has good extension at $N$, and $G_N$ has good isotropy at $\iota_N$.
			Then $G_N$ has good isotropy at $\nu$ and $G$ has good isotropy at $(N,\nu)$.
	\end{enumerate}
\end{lemma}

\begin{proof}
	In order to show that $\mathfrak A$ is an augmentation functor it remains to show that the $G_N$ action on $\mathfrak A(N)=\mathfrak D(N)$ is smooth.
	To this end, suppose $g\colon Z\to G_N$ is smooth when considered as a map into $G$.
	Then the map $Z\times N\to N$, $(z,x)\mapsto g(z)\cdot x$ is smooth, for $N$ is initial in $M$ and the $G$ action on $M$ is assumed to be smooth.
	Hence, the corresponding map $Z\to\Diff(N)$, $z\mapsto g(z)|_N$ is smooth.
	Since the $\Diff(N)$ action on $\mathfrak D(N)$ is smooth, we conclude that $Z\times\mathfrak D(N)\to\mathfrak D(N)$, $(z,\nu)\mapsto\mathfrak D(g(z)|_N)\cdot\nu$ is smooth.
	Hence, $\mathfrak A$ is an augmentation functor.

	(a) Let $U\subseteq G\cdot N$ be an open subset and suppose $u\colon U\to G$ is a local smooth section, i.e., $u(\tilde N)\cdot N=\tilde N$ for all $\tilde N\in U$.
	Let $\varphi\colon S\to N$ be a diffeomorphism.
	Then $U\to\Emb(S,M)|_U$, $\tilde N\mapsto u(\tilde N)\circ\varphi$ is a section of $\Emb(S,M)$ over $U$.
	Hence,
	\[
		U\times\mathfrak D(N)\to\Emb(S,M)|_U\times_{\Diff(S)}\mathfrak D(S),\quad
		(\tilde N,\hat\nu)\mapsto\bigl[u(\tilde N)\circ\varphi,\mathfrak D(\varphi^{-1})\cdot\hat\nu\bigr]
	\]
	is a diffeomorphism.
	Composing this with the canonical identifications in \eqref{E:abcd} and \eqref{E:aug.deco} we obtain
	\[
		U\times\mathfrak D(N)\to\Gr^{\aug}(M)|_U,\quad(\tilde N,\hat\nu)\mapsto\bigl(\tilde N,\mathfrak D(u(\tilde N))\cdot\hat\nu\bigr),
	\]
	a chart for the smooth structure on $\Gr^{\aug}(M)|_{G\cdot N}$, cf.~\eqref{E:triv.aug}.
	Hence, \eqref{E:aug.deco} is a diffeomorphism.

	(b) By assumption there exists a diffeomorphism $\varphi\colon S\to N$ such that $\mathfrak D(\varphi)\cdot\mu=\nu$.
	Hence, $\Diff(N)$ has good orbit at $\nu\in\mathfrak D(N)$ too.
	Composing local smooth sections for the $\Diff(N)$ action on $\Diff(N)\cdot\nu$ with a local smooth section of \eqref{E:homo.sigma} we obtain local smooth sections for the $G_N$ action on $G_N\cdot\nu$.

	(c) Combining (b) with Lemma~\ref{L:aug2} we see that $G\cdot(N,\nu)\to G\cdot N$ is a smooth fiber bundle.
	According to Lemma~\ref{L:deco1} $\Gr^{\deco}_{S,\mu}(M)|_{G\cdot N}\to G\cdot N$ is a smooth fiber bundle also.
	Suppose $(\tilde N,\tilde\nu)\in G\cdot(N,\nu)$.
	Then the fiber over $\tilde N$ of the former bundle, $G_{\tilde N}\cdot\tilde\nu$, is open and closed in the fiber of the latter bundle, $\Diff(\tilde N)\cdot\tilde\nu$ according to (b).
	Hence, the left hand side in \eqref{E:ggg} is open and closed in the right hand side.
	This inclusion is a diffeomorphism since both are initial splitting smooth submanifolds in $\Gr^{\aug}(M)_{G\cdot N}$ and $\Gr^{\deco}_S(M)|_{G\cdot N}$, respectively, cf.~Lemmas~\ref{L:deco1} and \ref{L:aug2}, and \eqref{E:aug.deco} is a diffeomorphism.

	(d) 
	Since $\Diff(S)$ has good isotropy at $\mu$, $\Diff(N)$ has good isotropy at $\nu$ too.
	Hence, $\Diff(N,\nu)$ is a splitting Lie subgroup in $\Diff(N)$.
	Since $G$ has good isotropy at $N$, $G_N$ is a splitting Lie subgroup in $G$.
	According to Lemma~\ref{L:geom2} the map $G_N\to\Diff(N)$ is a smooth fiber bundle.
	As $G_{(N,\nu)}$ is the preimage of $\Diff(N,\nu)$ under this map, we conclude that $G_{(N,\nu)}$ is a splitting Lie subgroup of $G_N$ and, thus, a splitting Lie subgroup of $G$ also.
\end{proof}

\section*{Acknowledgements}
	Part of the work was carried out during a research stay of the authors at the Erwin Schr\"odinger
	International Institute for Mathematics and Physics, specifically during the program ``Infinite-dimensional Geometry: Theory and Applications''.


\end{document}